\documentclass{article}[12pt]

\usepackage{amsthm}
\usepackage{mathrsfs}
\usepackage{amsmath}
\usepackage{amsfonts}
\usepackage{enumerate}
\usepackage{amssymb}
\usepackage{dsfont}
\usepackage{color}
\usepackage{mathrsfs}
\usepackage{systeme}

\usepackage[natbibapa]{apacite}
\usepackage{float}

\usepackage{emptypage}

\usepackage{hyperref}                                                                                                                 
 \hypersetup{colorlinks,citecolor=blue,filecolor=blue,linkcolor=blue,urlcolor=blue}
\usepackage[pdftex]{graphicx}
\usepackage{float}
\usepackage{enumitem}

\def\G{\mathcal{G}}
\def\P{\mathbb{P}}
\def\R{\mathbb{R}}
\def \F {\mathcal{F}}
\def\E{\mathbb{E}}

\def \1nd{\mathds{1}}

\def \I{\mathbb{I}}

\def\et{\widetilde{e}_{\theta}}

\def \V{V^{(\theta)}}
\def \G{G^{(\theta)}}

\def \b{b^{(\theta)}}
\def \g{g^{(\theta)}}
\def \hth{h^{(\theta)}}
\def \ct{c^{(\theta)}}
\def \U{U^{(\theta)}}

\def \m{m_{\theta}}
\def \dd{\textup{d}}

\newtheorem{deff}{Definition}[section]
\newtheorem{thm}[deff]{Theorem}

\newtheorem{lemma}[deff]{Lemma}
\newtheorem{cor}[deff]{Corollary}
\newtheorem{rem}[deff]{Remark}

\definecolor{deeplilac}{rgb}{0.6, 0.33, 0.73}
\definecolor{newultramarine}{rgb}{0.07, 0.04, 0.56}

\title{Predicting the Last Zero before an exponential time of a Spectrally Negative L\'evy Process}
\author{Erik J. Baurdoux\footnote{Department of Statistics, London School of Economics and Political Science. Houghton street, {\sc London, WC2A 2AE, United Kingdom.} E-mail: e.j.baurdoux@lse.ac.uk} \quad \& \quad Jos\'e M. Pedraza\footnote{Department of Statistics, London School of Economics and Political Science. Houghton street, {\sc London, WC2A 2AE, United Kingdom.} E-mail: j.m.pedraza-ramirez@lse.ac.uk}  }
\date{\today}

\begin{document}

\maketitle

\begin{abstract}

    Given a spectrally negative L\'evy process, we predict, in a $L_1$ sense, the last passage time of the process below zero before an independent exponential time. This optimal prediction problem generalises \cite{baurdoux2018predicting} where the infinite horizon problem is solved.  Using a similar argument as that in \cite{urusov2005property}, we show that this optimal prediction problem is equivalent to solving an optimal prediction problem in a finite horizon setting. Surprisingly (unlike the infinite horizon problem) an optimal stopping time is based on a curve that is killed at the moment the mean of the exponential time is reached. That is, an optimal stopping time is the first time the process crosses above a non-negative, continuous and non-increasing curve depending on time. This curve and the value function are characterised as a solution of a system of non-linear integral equations which can be understood as a generalisation of the free boundary equations (see e.g. \cite{peskir2006optimal} Chapter IV.14.1) in the presence of jumps. As an example, we calculate numerically such curve in the Brownian motion case and a compound Poisson process with exponential sized jumps perturbed by a Brownian motion.
\end{abstract}



\section{Introduction}
The study of last exit times has received much attention in several areas of applied probability, e.g. risk theory, finance and reliability in the past few years. Consider the Cram\'er--Lundberg process, a process consisting of a deterministic drift and a compound Poisson process with only negative jumps (see Figure \ref{Cramerlundberg}), which is typically used to model the capital of an insurance company. Of particular interest is the moment of ruin, $\tau_0$ which is defined to refer to the first moment when the process becomes negative. Within the framework of the insurance company having sufficient funds to endure negative capital for a considerale amount of time, another quantity of interest is the last time, $g$ that the process is below zero. In a more general setting, we can consider a spectrally negative L\'evy process instead of the classical risk process. Several studies, for example \cite{baurdoux2009last} and \cite{chiu2005passage} studied the Laplace transform of the last time before an exponential time that a spectrally negative L\'evy process is below some given level.

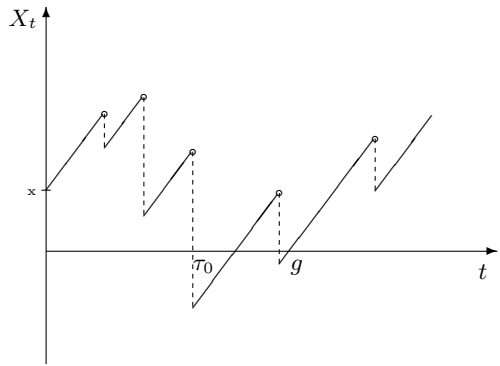
\begin{figure}[H]
\begin{center} 
\setlength{\unitlength}{.25cm} 
\centering 
\begin{picture}(30,20) 
 
\put(2,0){\vector(0,1){19}} 
\put(2,6){\vector(1,0){24}}
\put(0,18){\small {$X_t$}}
\put(25,4.5){\small{$t$}}

\put(1,9){\tiny{x}}
\put(1.75,9.25){\line(1,0){.5}}

\put(2,9.25){\line(3,4){3}}
\put(5.1,13.3){\circle{.3}}	
\multiput(5.1,13.3)(0,-.5){4}{\line(0,-1){.25}}

\put(5.1,11.5){\line(3,4){2}}
\put(7.2,14.2){\circle{.3}}
\multiput(7.2,14.2)(0,-.5){13}{\line(0,-1){.25}}


\put(7.2,7.9){\line(3,4){2.5}}
\put(9.8,11.3){\circle{.3}}
\multiput(9.8,11.3)(0,-.5){17}{\line(0,-1){.25}}

\put(9.8,5){\small{$\tau_0$}}

\put(15,5){\small{$g$}}

\put(9.8,3){\line(3,4){4.5}}
\put(14.4,9.1){\circle{.3}}
\multiput(14.4,9.1)(0,-.5){8}{\line(0,-1){.25}}

\put(14.4,5.3){\line(3,4){5}}
\put(19.5,12){\circle{.3}}
\multiput(19.5,12)(0,-.5){6}{\line(0,-1){.25}}

\put(19.5,9.2){\line(3,4){3}}

\end{picture}
\end{center} 
\caption{Cram\'er--Lundberg process with $\tau_0$, the moment of ruin and $g$, the last zero.}
\label{Cramerlundberg}
\end{figure}

Last passage time is increasingly becoming a vital factor in financial modeling as shown in \cite{madan2008black} and \cite{madan2008option} where the authors concludes that the price of a European put and call options, modelled by non-negative and continuous martingales that vanish at infinity, can be expressed in terms of the probability distributions of some last passage times.

Another application of last passage times is in degradation models. \cite{paroissin2013first} propose a spectrally positive L\'evy process to model the ageing of a device in which they consider a subordinator perturbed by an independent Brownian motion. A motivation for considering this model is that the presence of a Brownian motion can model small repairs of the device and the jumps represent major deterioration.  In the literature, the failure time of a device is defined as the first hitting time of a critical level $b$. An alternative approach is to consider instead, the last time that the process is under the level $b$ since the paths of this process are not necessarily monotone and this allows the process to return below the level $b$ after it goes above $b$.\\

The aim of this work is to predict the last time a spectrally negative L\'evy process is below zero before an independent exponential time where the terms ``to predict" are understood to mean to find a stopping time that is closest (in $L^1$ sense) to this random time. This problem is an example of the optimal prediction problems which have been widely investigated by many. \cite{Graversen2001} predicted the value of the ultimate maximum of a Brownian motion in a finite horizon setting whereas \cite{shiryaev2009} focused on the last time of the attainment of the ultimate maximum of a (driftless) Brownian motion and proceeded to show that it is equivalent to predicting the last zero of the process in this setting. The work of the latter was generalised by \cite{du2008predicting} for a linear Brownian motion. \cite{bernyk2011predicting} studied the time at which a stable spectrally negative L\'evy process attains its ultimate supremum in a finite horizon of time and this was later generalised by \cite{baurdoux2014predicting} for any L\'evy process in infinite horizon of time. Investigations on the time of the ultimate minimum and the last zero of a transient diffusion process were carried out by \cite{glover2013three} and \cite{glover2014optimal} respectively within a subclass of functions. \\

In \cite{baurdoux2018predicting} the last zero of a spectrally negative L\'evy process in a infinite horizon setting is predicted. It is shown that the an optimal stopping time that minimises the last zero of a spectrally negative L\'evy process with drift is the first time the L\'evy process crosses above a fixed level $a^*\geq 0$ which is characterised in terms of the cumulative distribution of the overall infimum of the process. As is in the case in the Canadisation of American type options (see e.g. \cite{carr1998randomization}), given the memoryless property of the exponential distribution, one would expect that the generalisation of the aforementioned problem to an exponential time horizon would result in an infinite horizon optimal prediction problem, and hence have a non time dependent solution. However, it turns out that this is not the case. Indeed, we show the existence of a continuous, non-increasing and non-negative boundary such that an optimal stopping time is given by the first passage time, before the median of the exponential time, above such curve. The proof relies on solving an equivalent (finite horizon) optimal stopping problem that depends on time and the process itself. Moreover, based on the ideas of \cite{du2008predicting} we characterise the boundary and the value function as the unique solutions of a system of non-linear integral equations. Such system can be thought as a generalisation of the free boundary equation (see e.g. \cite{peskir2006optimal}, Section 14) allowing for the presence of jumps. We consider two examples where numerical calculations are implemented to find the optimal boundary. \\

This paper is organised as follows. In Section \ref{sec:Prerequisites} we introduce some important notation regarding L\'evy processes and we outline some known fluctuation identities that will be useful later. We then formulate the optimal prediction problem and we prove that it is equivalent to an optimal stopping problem. Section \ref{sec:optimalstoppingproblem} is dedicated to the solution of the optimal stopping problem. The main result of this paper is stated in Theorem \ref{thm:Vsatisfiesequation} and its proof is detailed in Section \ref{sec:proof}. The last section makes use of Theorem \ref{thm:Vsatisfiesequation} to find numerical solution of the optimal stopping problem for the Brownian motion with drift case and a compound Poisson process perturbed by a Brownian motion.

\section{Prerequisites and Formulation of the Problem}
\label{sec:Prerequisites}

We start this section by introducing some important notations and we give an overview of some fluctuation identities of spectrally negative L\'evy processes. Readers can refer to \cite{bertoin1998levy}, \cite{sato1999levy} or \cite{kyprianou2014fluctuations} for more details about L\'evy processes.\\

A L\'evy process $X=\{X_t,t\geq 0 \}$ is an almost surely c\`adl\`ag process that has independent and stationary increments such that $\P(X_0=0)=1$. Every L\'evy process $X$ is also a strong Markov $\mathbb{F}$-adapted process. For all $x\in \R$, denote $\P_x$ as the law of $X$ when started at the point $x\in \R$, that is, $\E_x(\cdot)=\E(\cdot|X_0=x)$. Due to the spatial homogeneity of L\'evy processes, the law of $X$ under $\P_x$ is the same as that of $X+x$ under $\P$.\\

Let $X$ be a spectrally negative L\'evy process, that is, a L\'evy process starting from $0$ with only negative jumps and non-monotone paths,  defined on a filtered probability space $(\Omega,\F, \mathbb{F}, \P)$ where $\mathbb{F}=\{\F_t,t\geq 0 \}$ is the filtration generated by $X$ which is naturally enlarged (see Definition 1.3.38 in \cite{bichteler2002stochastic}). We suppose that $X$ has L\'evy triplet $(\mu,\sigma, \Pi)$ where $\mu \in \R$, $\sigma\geq 0$ and $\Pi$ is a measure  (L\'evy measure) concentrated on $(-\infty,0)$ satisfying $\int_{(-\infty,0)} (1\wedge x^2)\Pi(\dd x)<\infty$. \\

Let $\psi$ be the Laplace exponent of $X$ defined as

\begin{align*}
\psi(\beta):=\log(\E(e^{\beta X_1})).
\end{align*}

Then $\psi$ exists in $\R_+$, it is strictly convex and infinitely differentiable with $\psi(0)=0$ and $\psi(\infty)=\infty$. From the L\'evy--Khintchine formula,
we know that $\psi$ takes the form

\begin{align*}
\psi(\beta)=-\mu \beta +\frac{1}{2} \sigma^2 \beta^2 + \int_{(-\infty,0)}(e^{\beta x}-1-\beta x \I_{\{ x>-1\}}) \Pi(\dd x)
\end{align*}
for all $\beta\geq 0$. Moreover, from L\'evy--It\^o decomposition, we know that there exists a Brownian Motion $B$ and an independent Poisson random measure on $\R_+\times \R$ with intensity $\dd t\times \Pi(\dd y)$ such that for each $t\geq 0$,
\begin{align*}
X_t=-\mu t+\sigma B_t+\int_{[0,t]} \int_{(-\infty,-1)} y N(\dd s, \dd y)+\int_{[0,t]} \int_{(-1,0)} y[N(\dd s, \dd y)-\Pi(\dd y) \dd s].
\end{align*}
Denote $\tau_a^+$ as the first time the process $X$ is above the level $a \in \R$, i.e.,
\begin{align*}
\tau_a^+=\inf\{ t>0: X_t>a\}.
\end{align*}
Then it can be shown that its Laplace transform is given by 

\begin{align}
\label{eq:laplacetransformtau0+}
\E(e^{-q\tau_a^+}\I_{\{\tau_a^+<\infty \}})=e^{-\Phi(q)a},
\end{align}
where $\Phi$ corresponds to the right inverse of $\psi$, which is defined by

\begin{align*}
\Phi(q)=\sup\{ \theta \geq 0: \psi(\theta)=q \}
\end{align*}  
for any $q\geq 0$.\\

 Now we introduce the scale functions.  This family of functions is the key to the derivation of fluctuation identities for spectrally negative L\'evy processes. The notation used is mainly based on \cite{kyprianou2011theory} and  \cite{kyprianou2014fluctuations} (see Chapter 8). For $q\geq 0$, the function $W^{(q)}$ is such that $W^{(q)}=0$ for $x<0$ and $W^{(q)}$ is characterised on $[0,\infty)$ as a strictly increasing and continuous function whose Laplace transform satisfies 

\begin{align*}
\int_0^{\infty} e^{-\beta x} W^{(q)}(x)\dd x=\frac{1}{\psi(\beta)-q}, \qquad \text{for } \beta>\Phi(q).
\end{align*}
We further define the function $Z^{(q)}$ by
\begin{align*}
Z^{(q)}(x)=1+q\int_0^x W^{(q)}(y)\dd y.
\end{align*}
Denote $\tau_0^-$ as the first passage time of $X$ of the set $(-\infty,0)$, that is,
\begin{align*}
\tau_0^-=\inf\{t>0: X_t<0\}.
\end{align*}
It turns out that the Laplace transform of $\tau_0^-$ can be written in terms of the scale functions. Specifically,
\begin{align}
\label{eq:laplacetransformoftau0-}
\E_x(e^{-q\tau_0^-} \I_{\{\tau_0^-<\infty \}})=Z^{(q)}(x)-\frac{q}{\Phi(q)}W^{(q)}(x)
\end{align}
for all $q\geq 0$ and $x\in \R$. 
It can be shown that the paths of $X$ are of finite variation if only if 
\begin{align*}
\sigma=0 \qquad\text{and} \qquad \int_{(-1,0)}y \Pi(\dd y)<\infty.
\end{align*}
In such case, we may write 
\begin{align*}
\psi(\lambda)=\delta \lambda -\int_{(-\infty,0)}(1-e^{\lambda y})\Pi(dy),
\end{align*}
where
\begin{align}
\label{eq:definitionofdeltafinitevariation}
\delta:=-\mu-\int_{(-1,0)}x\Pi(\dd x).
\end{align}
Note that monotone processes are excluded from the definition of spectrally negative L\'evy processes so we assume that $\delta>0$ when $X$ is of finite variation.
%
The value of $W^{(q)}$ at zero depends on the path variation of $X$. In the case where $X$ is of infinite variation we have that $W^{(q)}(0)=0$, otherwise 

\begin{align}
\label{eq:Watzero}
W^{(q)}(0)=\frac{1}{\delta}.
\end{align}
For any $a \in \R$ and $q\geq 0$, the $q$-potential measure of $X$ killed upon entering the set $[a,\infty)$ is absolutely continuous with respect to the Lebesgue measure. A density is given for all $x,y\leq a$ by
\begin{align}
\label{eq:qpotentialmeasurekilledonexita}
\int_0^{\infty} e^{-qt} \P_x\left( X_t \in \dd y,t<\tau_a^+ \right) \dd t=[e^{-\Phi(q)(a-x)}W^{(q)}(a-y)-W^{(q)}(x-y) ]\dd y.
\end{align}

Let $g_{\theta}$ be the last passage time below zero before an exponential time, i.e.

\begin{align}
\label{eq:lastzero}
g_{\theta}=\sup\{0 \leq t \leq \widetilde{e}_{\theta}:X_t\leq 0 \},
\end{align}
where $\widetilde{e}$ is an exponential random variable with parameter $\theta \geq 0$. Here, we use the convention that an exponential random variable with parameter $0$ is taken to be infinite with probability $1$. In the case of $\theta=0$, we simply denote $g=g_0$. 

Note that $g_{\theta} \leq \et<\infty$ $\P$-a.s. for all $\theta > 0$. However, in the case where $\theta=0$, $g$ could be infinite. Therefore, we assume that $\theta>0$ throughout this paper. Moreover, we have that $g_{\theta}$ has finite moments for all $ \theta>0$. 


\begin{rem}
\label{rem:lastzero}
Since $X$ is a spectrally negative L\'evy process, we can exclude the case of a compound Poisson process and hence the only way of exiting the set $(-\infty,0]$ is by creeping upwards. This tells us that $X_{g_{\theta}-}=X_{g_{\theta}}=0$ in the event of $\{g_{ \theta}<\et\}$ and that $g_{\theta}=\sup\{0\leq  t \leq  \et: X_t<0\}$ holds $\P$-a.s.
\end{rem}
Clearly, up to any time $t\geq 0$ the value of $g_{\theta}$ is unknown (unless $X$ is trivial), and it is only with the realisation of the whole process that we know that the last passage time below $0$ has occurred. However, this is often too late: typically, at any time $t\geq 0$, we would like to know how close we are to the time $g_{\theta}$ so we can take some actions based on this information. We search for a stopping time $\tau_*$ of $X$ that is as ``close'' as possible to $g_{\theta}$. Consider the optimal prediction problem

\begin{align}
\label{eq:optimalprediction}
V_*=\inf_{\tau \in \mathcal{T} } \E(|g_{\theta}-\tau|),
\end{align}
where  $\mathcal{T}$ is the set of all stopping times.  \\

Note that the random time $\g_{\theta}$ is only $\mathbb{F}$ measurable so it is not immediately obvious how to solve the optimal prediction problem by using the theory of optimal stopping. Hence, in order to find the solution (and hence prove Theorem \ref{thm:solutionoftheOPP}) we solved an equivalent optimal stopping problem. In the next Lemma we establish an equivalence between the optimal prediction problem \eqref{eq:optimalprediction} and an optimal stopping problem. This equivalence is mainly based on the work of \cite{urusov2005property}. 
\begin{lemma}
\label{lemma:equivalentoptimalstoppingproblem}
Suppose that $\{X_t, t\geq 0\}$ is a spectrally negative L\'evy process. Let $g_{\theta}$ be the last time that $X$ is below the level zero before an exponential time $\widetilde{e}_{\theta}$ with $\theta > 0$, as defined in \eqref{eq:lastzero}. Consider the optimal stopping problem given by 
\begin{align}
\label{eq:optimalstopping}
V=\inf_{\tau\in \mathcal{T}} \E\left( \int_0^{\tau} G^{(\theta)}(s,X_s)\dd s\right),
\end{align}
where the function $G^{(\theta)}$ is given by $G^{(\theta)}(s,x)=1+2e^{-\theta s}\left[ \frac{\theta}{\Phi(\theta)} W^{(\theta)}(x)-Z^{(\theta)}(x)\right]$ for all $x\in \R$. Then the stopping time which minimises \eqref{eq:optimalprediction} is the same which minimises \eqref{eq:optimalstopping}. In particular,

\begin{align}
V_*=V+\E(g_{\theta}).
\end{align}
\end{lemma}

\begin{proof}
Fix any stopping time $\tau \in \mathcal{T}$. We have that

\begin{align*}
|g_{\theta}-\tau|= \int_0^{\tau}[2\I_{\{ g_{\theta}\leq s\}}-1]\dd s+g_{\theta}.
\end{align*}
From Fubini's theorem and the tower property of conditional expectations, we obtain 

\begin{align*}
\E\left[\int_0^{\tau} \I_{\{ g_{\theta}\leq s\}} \dd s\right]
&=\E\left[ \int_0^{\infty}\I_{\{s<\tau \}}  \E[\I_{\{ g_{\theta}\leq s\}}|\F_s] \dd s \right]\\
&=\E\left[ \int_0^{\tau}  \P( g_{\theta}\leq s |\F_s)\dd s \right].
\end{align*}
Note that in the event of $\{\widetilde{e}_{\theta}\leq s\}$, we have $g_{\theta}\leq s$ so that

\begin{align*}
\P(g_{\theta}\leq s|\F_s)=1-e^{-\theta s}+\P(g_{\theta}\leq s,\et>s|\F_s ).
\end{align*}
On the other hand for $\{\widetilde{e}_{\theta } >s\}$, as a consequence of Remark \ref{rem:lastzero}, the event $\{g_{\theta}\leq s\}$ is equal to $\{ X_u \geq 0 \text{ for all } u\in [s,\widetilde{e}_{\theta}]\}$ (up to a $\P$-null set). Hence, we get that for all $s\geq 0$ that 
\begin{align*}
\P(g_{\theta}\leq s, \et>s|\F_s)&=\P(X_u \geq  0 \text{ for all } u\in [s,\widetilde{e}_{\theta}],\et>s|\F_s )\\
&=\P\left(\inf_{0 \leq  u \leq \widetilde{e}_{\theta}-s} X_{u+s} \geq  0,\et>s|\F_s\right) \\
&= e^{-\theta s}  \P_{X_s}\left( \underline{X}_{\et } \geq 0\right) ,
\end{align*}
where the last equality follows from the lack of memory property of the exponential distribution and the Markov property for L\'evy process. Hence, we have that 
\begin{align*}
\P(g_{\theta}\leq s, \et>s|\F_s)&=e^{-\theta s}F^{(\theta)}(X_s),
\end{align*}
where for all $x\in \R$, $F^{(\theta)}(x)=\P_x(\underline{X}_{\et } \geq 0)$. Then, since $\et$ is independent of $X$, we have that for $x\in \R$,

\begin{align*}
F^{(\theta)}(x)=\P_x(\underline{X}_{\et}\geq  0)=\P_x(\et < \tau_0^-)=1-\E_x(e^{-\theta \tau_0^-} \I_{\{\tau_0^-<\infty\}})=\frac{\theta}{\Phi(\theta)}W^{(\theta)}(x)-Z^{(\theta)}(x)+1,
\end{align*}
where the last equality follows from equation \eqref{eq:laplacetransformoftau0-}. Thus,

\begin{align*}
\P(g_{\theta}\leq s |\F_s)&=1-e ^{-\theta s}+ e^{-\theta s}\left[ \frac{\theta}{\Phi(\theta)} W^{(\theta)}(X_s)-Z^{(\theta)}(X_s)+1 \right]\\
&=1+e^{-\theta s}\left[ \frac{\theta}{\Phi(\theta)} W^{(\theta)}(X_s)-Z^{(\theta)}(X_s)\right].
\end{align*}
Therefore,
\begin{align*}
V_*&=\inf_{\tau\in \mathcal{T}} \E(|g_{\theta}-\tau|)\\
&=\E(g_{\theta})+\inf_{\tau \in \mathcal{T}} \E\left(\int_0^{\tau} [2\P(g_{\theta}\leq s|\F_s)-1]\dd s  \right)\\
&=\E(g_{\theta})+\inf_{\tau \in \mathcal{T}} \E\left(\int_0^{\tau} \left(1+2e^{-\theta s}\left[ \frac{\theta}{\Phi(\theta)} W^{(\theta)}(X_s)-Z^{(\theta)}(X_s)\right]\right)\dd s  \right).
\end{align*}
The conclusion holds.
\end{proof}
Note that evaluating $\theta=0$, the function $G^{(0)}$ coincides with the gain function found in \cite{baurdoux2018predicting} (see Lemma 3.2 and Remark 3.3). In order to find the solution to the optimal stopping problem \eqref{eq:optimalstopping} (and hence \eqref{eq:optimalprediction}), we extend its definition to L\'evy process (and hence strong Markov process) $\{(t,X_t),t\geq 0\}$ in the following way. Define the function $V: \R_+\times \R \mapsto \R$ as 
\begin{align}
\label{eq:optimalstoppingforallx}
V^{(\theta)}(t,x)=\inf_{\tau\in \mathcal{T} }\E_{t,x}\left( \int_0^{\tau} G^{(\theta)}(s+t,X_{s+t})\dd s\right)= \inf_{\tau \in \mathcal{T}}\E\left( \int_0^{\tau} G^{(\theta)}(s+t,X_s+x)\dd s \right).
\end{align}
So that  
\begin{align*}
V_*=V^{(\theta)}(0,0)+\E(g_{\theta}).
\end{align*}
The next theorem states the solution of the optimal stopping theorem \eqref{eq:optimalstoppingforallx} and hence the solution of \eqref{eq:optimalprediction}.
\begin{thm}
\label{thm:solutionoftheOPP}
Let $\{X_t, t\geq 0\}$ be any spectrally negative L\'evy process and $\widetilde{e}_{\theta}$ an exponential random variable with parameter $\theta > 0$ independent of $\mathbb{F}$. There exists a non increasing and continuous curve $\b:[0,\m]  \mapsto \R_+$ such that $\b\geq \hth$, where $h^{(\theta)}(t):=\inf\{ x \in \R: G^{(\theta)}(t,x)\geq 0\}$ and the infimum in \eqref{eq:optimalstoppingforallx} is attained by the stopping time
\begin{align*}
\tau_D=\inf\{ t \in [0,\m]: X_t \geq \b(t) \},
\end{align*}
where $\m=\log(2)/\theta$. Moreover, the function $\b$ is uniquely characterised as in Theorem \ref{thm:Vsatisfiesequation}.  
\end{thm}
Note that the proof of Theorem \ref{thm:solutionoftheOPP} is rather long and hence is split in a series of lemmas. We dedicate Section \ref{sec:optimalstoppingproblem} for that purpose.

\section{Solution to the optimal stopping problem}
\label{sec:optimalstoppingproblem}
In this section we solve the optimal stopping problem \eqref{eq:optimalstoppingforallx}. The proof relies on showing that $\tau_D$ defined in Theorem \ref{thm:solutionoftheOPP} is indeed an optimal stopping time by using the general theory of optimal stopping and properties of the function $\V$. Hence, some properties of $\b$ are derived. The main contribution of this section (Theorem \ref{thm:Vsatisfiesequation}) characterises $\V$ and $\b$ as the unique solution of a non-linear system of integral equations within a certain family of functions.\\

Recall that the $\V$ is given by 
\begin{align*}
V^{(\theta)}(t,x)=\inf_{\tau\in \mathcal{T} }\E_{t,x}\left( \int_0^{\tau} G^{(\theta)}(s+t,X_{s+t})\dd s\right).
\end{align*}
From the proof of Lemma \ref{lemma:equivalentoptimalstoppingproblem} we note that $G^{(\theta)}$ can be written as,
\begin{align*}
G^{(\theta)}(s,x)=1+2e^{-\theta s} [F^{(\theta)}(x)-1],
\end{align*}
where $F^{(\theta)}$ is the distribution function of the positive random variable $-\underline{X}_{\et}$ given by
\begin{align}
\label{eq:definitionofFtheta}
F^{(\theta)}(x)=\frac{\theta}{\Phi(\theta)}W^{(\theta)}(x)-Z^{(\theta)}(x)+1
\end{align}
for all $x\in \R$. Now we give some intuitions about the function $G^{(\theta)}$. Recall that for all $\theta \geq 0$, $W^{\theta}$ and $Z^{(\theta)}$ are continuous and strictly increasing functions on $[0,\infty)$ such that $W^{(\theta)}(x)=0$ and $Z^{(\theta)}(x)=1$ for $x \in (-\infty,0)$. From the above, equation \eqref{eq:definitionofFtheta} and from the fact that $F^{(\theta)}$ is a distribution function, we have that for a fixed $t\geq 0$, the function $x\mapsto G^{(\theta)}(t,x)$ is strictly increasing and continuous in $[0,\infty)$ with a possible discontinuity at $0$ depending on the path variation of $X$. Moreover, we have that $\lim_{x\rightarrow \infty} G^{(\theta)}(t,x)=1$ for all $t\geq 0$. For $x<0$ and $t\geq 0$, we have that the function $G^{(\theta)}$ takes the form $G^{(\theta)}(t,x)=1-2e^{-\theta s}$. Similarly, from the fact that $F^{(\theta)}(x)-1\leq 0$ for all $x\in \R$, we have that for a fixed $x\in \R$ the function $t\mapsto G^{(\theta)}(t,x)$ is continuous and strictly increasing on $[0,\infty)$. Furthermore, from the fact that $0\leq F^{(\theta)}(x) \leq 1$, we have that the function $G$ is bounded by

\begin{align}
\label{eq:boundariesofG}
1-2e^{-\theta t}\leq G^{(\theta)}(x,t)\leq 1
\end{align}
which implies that $|G^{(\theta)}|\leq 1$. Recall that $\m$ is defined as the median of the random variable $\et$, that is,
\begin{align*}
\m=\frac{\log(2)}{\theta}.
\end{align*}

Hence from \eqref{eq:boundariesofG} we have that $G^{(\theta)}(t,x)\geq 0$ for all $x\in \R$ and $t\geq \m$. The above observations tell us that, to solve the optimal stopping problem \eqref{eq:optimalstoppingforallx}, we are interested in a stopping time such that before stopping, the process $X$ spends most of its time in the region where $\G$ is negative, taking into account that $(t,X)$ can live in the set $\{ (s,x) \in \R_+\times \R: \G(s,x)>0\}$ and then return back to the set $\{ (s,x) \in \R_+\times \R: \G(s,x)\leq 0 \}$. The only restriction that applies is that if a considerable amount of time has passed, then $\{ x \in \R: \G(s,x)>0\}=\R$ for all $s\geq \m$.\\

Recall that the function $h^{(\theta)}:\R_+\mapsto \R$ is defined as
\begin{align}
\label{eq:definitionofhth}
h^{(\theta)}(t)=\inf\{ x \in \R: G^{(\theta)}(t,x)\geq 0\}=\inf\{ x \in \R: F^{(\theta)}(x) \geq  1-\frac{1}{2}e^{\theta t} \} \qquad t\geq 0.
\end{align}
Hence, we can see that the function $h^{(\theta)}$ is a non-increasing continuous function on $[0,\m)$ such that $\lim_{t\uparrow \m} h^{(\theta)}(t)=0$ and $\hth(t)=-\infty$ for $t\in[\m,\infty)$. Moreover, from the fact that $G^{(\theta)}(t,x)<0$ for $(t,x)\in [0,\m)\times (-\infty,0)$, we have that $h^{(\theta)}(t)\geq 0$ for all $t\in [0,\m)$.\\

In order to characterise the stopping time that minimises \eqref{eq:optimalstoppingforallx}, we first derive some properties of the function $\V$.

\begin{lemma}
\label{lemma:basicpropertiesofV}
The function $\V$ is non-drecreasing in each argument. Moreover, $V^{(\theta)}(t,x) \in (-\m,0]$ for all $x\in \R$ and $t \geq 0$. In particular, $V^{(\theta)}(t,x)<0$ for any $t\geq 0$ with $x < h^{(\theta)}(t)$ and $\V(t,x)=0$ for all $(t,x)\in [\m,\infty)\times \R$.
\end{lemma}

\begin{proof}
First, note that $V^{(\theta)}\leq 0$ follows by taking $\tau \equiv 0$ in the definition of $\V$. Moreover, since $\G \geq 0$ on $[\m,\infty)\times \R$ we have that $\V $ vanishes on $[\m,\infty)\times \R$ . The fact that $\V$ is non-decreasing in each argument follows from the non-decreasing property of the functions $t \mapsto G^{(\theta)}(t,x)$ and $x \mapsto G^{(\theta)}(t,x)$ as well as the monotonicity of the expectation. Moreover, using standard arguments we can see that $\{(t,x)\in \R_+ \times \R: x<\hth(t)\}=\{(t,x)\in \R_+ \times \R: G^{(\theta)}(t,x)<0\} \subset \{ (t,x)\in \R_+\times \R:  V^{(\theta)}(t,x)<0\}$.\\

Next we will show that $V^{(\theta)}(t,x)>-\m$ for all $(t,x)\in [0,\m)\times \R$ and for all $\theta > 0$. Note that $t<\m$ if and only if $1-2e^{-\theta t}<0$. Then for all $(s,x)\in \R_+\times \R$ we have that
\begin{align*}
\G(s,x)\geq 1-2e^{-\theta s}\geq (1-2e^{-\theta s})\I_{\{s<\m \}}.
\end{align*}
Hence, for all $x\in \R$ and $t<\m$

\begin{align*}
\V(t,x) \geq \inf_{\tau\in \mathcal{T}} \E \left(\int_0^{\tau} (1-2e^{-\theta (s+t)})\I_{\{t+s<\m \}} \dd s \right) = - \sup_{\tau\in \mathcal{T}} \E \left(\int_0^{\tau} (2e^{-\theta (s+t)}-1)\I_{\{t+s<\m \}} \dd s \right).
\end{align*}
The term in the last integral is non-negative, so we obtain for all $t<\m$ and $x\in \R$ that

\begin{align*}
\V(t,x) \geq -\left(\int_0^{\infty}  (2e^{-\theta (s+t)}-1)\I_{\{t+s<\m \}}\dd s \right)
= -\left(\int_0^{\m-t}  (2e^{-\theta (s+t)}-1)\dd s \right)
>-\m.
\end{align*}

\end{proof}
By using a dynamic programming argument and the fact that $\V$ vanishes on the set $[0,\m)\times \R$ we can see that 
\begin{align*}
\V(t,x) &=\inf_{\tau\in \mathcal{T} }\E_{t,x}\left( \int_0^{\tau \wedge (\m-t)} G^{(\theta)}(s+t,X_{s+t})\dd s\right).
\end{align*}
so that (since $|\G|\leq 1$) we have that for all $t\geq 0$ and $x\in \R$,
\begin{align*}
\E_{t,x}\left( \sup_{s\geq 0} \left|  \int_0^{s \wedge (\m-t)} G^{(\theta)}(r+t,X_{r+t})\dd r \right|\right)<\infty.
\end{align*}
Moreover, as a consequence of the properties of $F^{(\theta)}$ we have that the function $\G$ is upper semi-continuous we can see that $\V$ is upper semi-continuous (since $\V$ is the infimum of upper semi-continuous functions). Therefore, from the general theory of optimal stopping (see \cite{peskir2006optimal} Corollary 2.9) we have that an optimal stopping time for \eqref{eq:optimalstoppingforallx} is given by
\begin{align}
\label{eq:definitionofTDexptime}
\tau_{D}=\inf\{t\geq 0: (t,X_t)\in D\},
\end{align}
where $D=\{(t,x)\in \R_+\times \R: \V(t,x)=0 \}$ is a closed set.

Hence, from Lemma \ref{lemma:basicpropertiesofV}, we derive that $D=\{ (t,x) \in \R_+ \times \R: x \geq b^{(\theta)}(t)\}$, where the function $b^{(\theta)}:\R_+ \mapsto \R$ is given by

\begin{align*}
b^{(\theta)}(t)=\inf\{x \in \R: (t,x)\in D \},
\end{align*} 
for each $t\geq 0$. It follows from Lemma \ref{lemma:basicpropertiesofV} that $b^{(\theta)}$ is non-increasing and $b^{(\theta)}(t)\geq h^{(\theta)}(t)\geq 0$ for all $ t\geq 0$. Moreover, $\b(t)=-\infty$ for $t \in [\m,\infty)$, since $V^{(\theta)}(t,x)=0$ for all $t\geq \m$ and  $x \in \R$, giving us $\tau_D \leq \m$. In the case that $t<\m$, we have that $b^{(\theta)}(t)$ is finitely valued as we will prove in the following Lemma.
\begin{lemma}
\label{lemma:finitnessofb}
Let $\theta> 0$. The function $\b$ is finitely valued for all $t\in [0,\m)$.
\end{lemma}

\begin{proof}
For any $\theta >0$ and fix $t\geq 0$, consider the optimal stopping problem,

\begin{align*}
\mathcal{V}_t^{(\theta)}(x)=\inf_{\tau \in \mathcal{T}_{\m -t}} \E_x\left(\int_0^{\tau}  [1+2e^{-\theta t}( F^{(\theta)}(X_s)-1)] \dd s \right), \qquad x\in \R,
\end{align*}
where $\mathcal{T}_{\m-t}$ is the set of all stopping times of $\mathbb{F}$ bounded by $\m-t$. From the fact that for all $s \geq  0$ and $x\in \R$, $G(s+t,x)\geq 1+2e^{-\theta t} (F^{(\theta)}(x)-1)$ and that $\tau_D\in \mathcal{T}_{\m-t}$ (under $\P_{t,x}$ for all $x\in \R$), we have that
\begin{align}
\label{eq:inequality}
\V(t,x)\geq \mathcal{V}_t^{(\theta)}(x)
\end{align}
for all $x\in \R$. Hence it suffices to show that there exists $\tilde{x}_t$ (finite) sufficiently large such that $\mathcal{V}_t^{(\theta)}(x)=0$ for all $x\geq \tilde{x}_t$.\\

It can be shown that an optimal stopping time for $\mathcal{V}_t^{(\theta)}$ is $\tau_{\mathcal{D}_t}$, the first entry time before $\m-t$ to the set $\mathcal{D}_t=\{ x\in \R: \mathcal{V}_t^{(\theta)}(x)=0 \}$. We proceed by contradiction, assume that $\mathcal{D}_t=\emptyset$, then $\tau_{\mathcal{D}_t}=\m-t$. Hence, by the dominated convergence theorem and the spatial homogeneity of L\'evy processes we have that 
\begin{align*}
0 \geq \lim_{x\rightarrow \infty} \mathcal{V}_t^{(\theta)}(x)=  \E\left(\int_0^{\m-t}  \lim_{x\rightarrow \infty} [1+2e^{-\theta t}( F^{(\theta)}(X_s+x)-1)] \dd s \right)=\m-t>0
\end{align*}
which is a contradiction. Therefore, we conclude that for each $t \geq 0$, there exists a finite value $\tilde{x}_t$ such that $\b(t)\leq \tilde{x}_t$. 
\end{proof}

\begin{rem}
From the proof of Lemma \ref{lemma:finitnessofb}, we find an upper bound of the boundary $\b$. Define, for each $t\in [0,\m)$, $u^{(\theta)}(t)=\inf\{x \in \R: \mathcal{V}^{(\theta)}_t(x)=0 \}$. Then it follows that $u^{(\theta)}$ is a non-increasing finite function such that
\begin{align*}
u^{(\theta)}(t)\geq \b(t)
\end{align*}
for all $t \in [0,\m)$.
\end{rem}

Next we show that the function $\V$ is continuous.
\begin{lemma}
\label{lemma:Vcontinuity}
The function $\V$ is continuous. Moreover, for each $x\in \R$, $t \mapsto \V(t,x)$ is Lipschitz on $\R_+$ and for every $t \in \R_+$, $x \mapsto \V(t,x)$ is Lipschitz on $\R$. 
\end{lemma}

\begin{proof}
First, we are showing that, for a fixed $t\geq 0$, the function $x \mapsto V^{(\theta)}(t,x)$ is Lipschitz on $\R$. Recall that if $t\geq \m$, $V^{(\theta)}(t,x)=0$ for all $x\in \R$ so the assertion is clear. Suppose that $t<\m$. Let $x, y \in  \R$ and define $\tau_{x}^*= \tau_{D(t,x)}=\inf\{s\geq 0: X_s+x\geq \b(s+t) \}$. Since $\tau_x^*$ is optimal in $V^{(\theta)}(t,x)$ (under $\P$) we have that 
\begin{align*}
V^{(\theta)}(t,y)-V^{(\theta)}(t,x)&\leq  \E\left(\int_0^{\tau_x^*} G^{(\theta)}(s+t,X_s+y)\dd s\right)-\E\left(\int_0^{\tau_x^*} G^{(\theta)}(s+t,X_s+x)\dd s\right)\\
&=\E\left(\int_0^{\tau_x^*} 2e^{-\theta (s+t)}[F^{(\theta)}(X_s+y)-F^{(\theta)}(X_s+x)]\dd s\right).
\end{align*}
Define the stopping time 
\begin{align*}
\tau_{b(0)-x}^+=\inf \{t \geq  0: X_t \geq b^{(\theta)}(0)-x\}.
\end{align*}
Then we have that $\tau_x^* \leq \tau_{b^{(\theta)}(0)-x}^+$ (since $b^{(\theta)}$ is a non-increasing function). From the fact 
that $F^{(\theta)}$ is non-decreasing, we obtain that for $\b(0)\geq y\geq x$, 
\begin{align*}
V^{(\theta)}(t,y)-V^{(\theta)}(t,x)
&\leq 2 \E\left(\int_0^{\tau_{b(0)-x}^+} e^{-\theta s}[F^{(\theta)}(X_s+y)-F^{(\theta)}(X_s+x)]\dd s\right).
\end{align*}
Using Fubini's theorem and a density of the potential measure of the process killed upon exiting $(-\infty,\b(0)]$ (see equation \eqref{eq:qpotentialmeasurekilledonexita}) we get that
\begin{align*}
V^{(\theta)}(t,y)-V^{(\theta)}(t,x)
&\leq  2 \int_{-\infty}^{\b(0)} [F^{(\theta)}(z+y-x)-F^{(\theta)}(z)]\int_0^{\infty}e^{-\theta s}   \P_x(X_s\in \dd z,\tau_{b(0)}^+>s)\dd s \\
&= 2\int_{-\infty}^{\b(0)} [F^{(\theta)}(z+y-x)-F^{(\theta)}(z)] \left[e^{-\Phi(\theta)(\b(0)-x)} W^{(\theta)}(\b(0)-z)-W^{(\theta)}(x-z) \right] \dd z\\
&\leq  2 e^{-\Phi(\theta)(\b(0)-x)} W^{(\theta)}(\b(0)-x+y)\int_{x-y}^{\b(0)} [F^{(\theta)}(z+y-x)-F^{(\theta)}(z)]   \dd z,
\end{align*}
where in the last inequality, we used the fact that $W^{(\theta)}$ is strictly increasing and non-negative and that $F^{(\theta)}$ vanishes at $(-\infty,0)$. By an integration by parts argument, we obtain that 
\begin{align*}
\int_{x-y}^{\b(0)} [F^{(\theta)}(z+y-x)-F^{(\theta)}(z)]   \dd z=(y-x)F^{(\theta)}(\b(0)+y-x).
\end{align*}
Moreover, it can be checked that (see \cite{kyprianou2011theory} lemma 3.3) the function $z\mapsto e^{-\Phi(\theta)(z)}W^{(\theta)}(z)$ is a continuous function in the interval $[0,\infty)$ such that 
\begin{align*}
\lim_{z\rightarrow \infty} e^{-\Phi(\theta)(z)}W^{(\theta)}(z)=\frac{1}{\psi'(\Phi(\theta))}<\infty.
\end{align*}

This implies that there exist a constant $M>0$ such that for every $z\in \R$, $ 0\leq e^{-\Phi(\theta)(z)}W^{(\theta)}(z)<M$. Then we obtain that for all $x\leq y \leq \b(0)$,

\begin{align*}
0\leq \V(t,y)-\V(t,x)\leq 2M (y-x) e^{\Phi(\theta) y} \leq 2M (y-x) e^{\Phi(\theta) \b(0)}.
\end{align*}
On the other hand, since $\b(0)\geq \b(t)$ for all $t\in [0,\m)$ we have that for all $(t,x) \in  [0,\m) \times [\b(0),\infty)$, $\V(t,x)=0$. Hence we obtain that for all $x,y \in \R$ and $t\geq 0$,
\begin{align}
\label{eq:LipschitzcontinuityofVonx}
|V^{(\theta)}(t,y)-V^{(\theta)}(t,x)| \leq 2 M |y-x| e^{\Phi(\theta) \b(0)}.
\end{align}
Therefore we conclude that for a fixed $t\geq 0$, the function $x \mapsto \V(t,x)$ is Lipschitz on $\R$.\\

Using a similar argument and the fact that the function $t \mapsto e^{-\theta t}$ is Lipschitz continuous on $[0,\infty)$ we can show that for any $s,t< \m$,
\begin{align*}
|V^{(\theta)}(s,x)-V^{(\theta)}(t,x)| \leq  2\theta \m|s-t|
\end{align*}
and therefore $t\mapsto \V(t,x)$ is Lipschitz continuous for all $x\in \R$.
%
%
%
%
%
\end{proof}
In order to derive more properties of the boundary $\b$, we first state some auxiliary results. Recall that if $f \in C_b^{1,2}(\R_+ \times \R)$, the set of real bounded $C^{1,2}$ functions on $ \R_+ \times \R$  with bounded derivatives, the infinitesimal generator of $(t,X)$ is given by 
\begin{align}
\label{eq:generatorofX}
\mathcal{A}_{(t,X)} (f)(t,x)&=\frac{\partial }{\partial t} f(t,x)-\mu \frac{\partial }{\partial x} f(t,x)+\frac{1}{2}\sigma^2 \frac{\partial^2}{\partial x^2} f (t,x)\nonumber\\
&\qquad+ \int_{(-\infty,0)}[f(t,x+y)-f(t,x)-y\I_{\{y>-1 \}} \frac{\partial}{\partial x} f(t,x)]\Pi(\dd y). 
\end{align}

Let $C=\R_+ \times \R  \setminus D=\{ (t,x) \in \R_+ \times \R: x<\b(t)\}$ be the continuation region. Then we have that the value function $\V$ satisfies a variational inequality in the sense of distributions. The proof is analogous to the one presented in \cite{lamberton2008critical} (see Proposition 2.5) so the details are omitted.
%
\begin{lemma}
\label{lemma:variationalcharacterization}
Fix $\theta > 0$. The distribution $\mathcal{A}_{(t,X)} \V  +\G$ is non-negative on $\R_+\times \R$. Moreover, we have that $ \mathcal{A}_{(t,X)} \V +\G=0$ on $C$.
\end{lemma}
We define a special function which is useful to prove the left-continuity of the boundary $\b$. For $\theta>0$, we define an auxiliary function in the set $D$. Let
\begin{align}
\label{eq:functionphitdef}
\varphi^{(\theta)}(t,x)= \int_{(-\infty,0)} \V(t,x+y) \Pi( \dd y)+\G(t,x), \qquad (t,x)\in  D.
\end{align}
From the fact that $\V$ vanishes on $D$ and that $\Pi$ is finite on sets of the form $(-\infty,-\varepsilon)$ for $\varepsilon>0$ we can see that $|\varphi^{(\theta)}(t,x)|<\infty$ for all $(t,x)\in D$. Moreover, by the Lemma above and the properties of $\V$ and $\G$, it can be shown that $\varphi$ is strictly positive, continuous and strictly increasing (in each argument) in the interior of $D$.\\

Now we are ready to give further properties of the curve $\b$ in the set $[0,\m)$. 
\begin{lemma}
The function $b^{(\theta)}$ is continuous on $[0,\m)$. Moreover we have that $\lim_{t\uparrow \m} \b(t)=0$ . 
\end{lemma}
\begin{proof}
The method proof of the continuity of $\b$ in $[0,\m)$ is heavyly based on the work of \cite{lamberton2008critical} (see Theorem 4.2, where the continuity of the boundary is shown in the American option context) so is omitted. \\

We then show that the limit holds. Define $\b(\m-):=\lim_{t\uparrow \m} \b(t)$. We obtain $\b(\m-)\geq 0$ since $\b(t) \geq \hth(t)\geq 0$ for all $t\in [0,\m)$. The proof is by contradiction so we assume that $\b(\m-)>0$. Note that for all $x \in \R$, we have that $\V(\m,0)=0$ and $\G(\m,x)=F^{(\theta)}(x)$. Moreover, we have that
\begin{align*}
\mathcal{A}_X (\V)+\G=- \partial_t \V \leq 0
\end{align*}
in the sense of distributions on $(0,\m)\times (0,\b(\m-) )$. Hence, by continuity, we can derive, for $t\in [0,\m)$, that $\mathcal{A}_X (\V)(t,\cdot)+\G(t,\cdot) \leq 0 $ on the interval $(0,\b(\m-))$. Hence, by taking $t\uparrow \m$ we obtain that
\begin{align*}
0\geq \lim_{t\uparrow \m} \mathcal{A}_X (\V)(t,\cdot)+\G(t,\cdot)=F^{(\theta)} >0
\end{align*}
in the sense of distributions, where we used the continuity if $\V$ and $\G$, the fact that $\V(\m,x)=0$ for all $x\in \R$ and that $F^{(\theta)}(x)>0$ for all $x>0$. Note that we have got a contradiction and we conclude that $\b(\m)=0$. 
\end{proof}
Define the value 
\begin{align}
\label{eq:definitionoftb}
t_b:=\inf\{t\geq 0: \b(t)\leq 0 \}.
\end{align}
Note that in the case where $X$ is a process of infinite variation, we have that the distribution function of the random variable $-\underline{X}_{\et}$, $F^{(\theta)}$, is continuous on $\R$, strictly increasing and strictly positive in the open set $(0,\infty)$ with $F^{(\theta)}(0)=0$. This fact implies that the inverse function of $F^{(\theta)}$ exists on $(0,\infty)$ and then function $\hth$ can be written for $t\in [0,\m)$ as

\begin{align*}
h^{(\theta)}(t)= (F^{(\theta)})^{-1}\left(  1-\frac{1}{2 } e^{\theta t}\right).
\end{align*}
Hence we conclude that $\hth(t)>0$ for all $t\in [0,\m)$. Therefore, when $X$ is a process of infinite variation, we have $\b(t)>0$ for all $t\in [0,\m)$ and hence $t_b=\m$. For the case of finite variation, we have that $t_b \in [0,\m)$ which implies that $\b(t)=0$ for all $t\in [t_b,\m)$ and $\b(t)>0$ for all $t\in [0,t_b)$. In the next lemma, we characterise its value.

\begin{lemma}
\label{lemma:characterizationoftb}
Let $\theta>0$ and $X$ be a process of finite variation. We have that for all $t\geq 0$ and $x\in \R$,
\begin{align*}
\int_{(-\infty,0)} [\V(t,x+y)-\V(t,x)]\Pi(\dd y)>-\infty.
\end{align*}
Moreover, for any L\'evy process, $t_b$ is given by
\begin{align}
\label{eq:characterisationoftb}
t_b=\inf \left\{ t\in [0,\m]:\int_{(-\infty,0)}\V_B(t,y)\Pi(\dd y) +\G(t,0)\geq 0 \right\},
\end{align}
where $\V_B$ is given by 
\begin{align*}
\V_B(t,y)= \E_{y}(\tau_0^+\wedge (\m-t))-\frac{2}{\theta}e^{-\theta t}[1-\E_{y}(e^{-\theta (\tau_0^+ \wedge (\m-t))})]
\end{align*}
for all $t\in [0,\m)$ and $y \in \R$.

\end{lemma}

\begin{proof}
Assume that $X$ is a process of finite variation. We first show that 
\begin{align*}
\int_{(-\infty,0)} [\V(t,x+y)-\V(t,x)]\Pi(\dd y)>-\infty    
\end{align*}
for all $t\geq 0$ and $x\in \R$. The case $t\geq \m$ is straightforward since $\V(t,x)=0$ for all $x\in \R$. The case $t<\m$ follows from the Lipschitz continuity of the mapping $x\mapsto \V(t,x)$, that $\Pi$ is finite on intervals away from zero and since $\int_{(-1,0)} y\Pi(\dd y)>-\infty$ when $X$ is of finite variation. Moreover, from Lemma \ref{lemma:variationalcharacterization}, we obtain that 
\begin{align*}
\int_{(-\infty,0)} [\V(t,x+y)-\V(t,x)] \Pi(\dd y)+\G(t,x)=- \frac{\partial}{\partial t} \V(t,x) -\delta  \frac{\partial}{\partial x} \V(t,x)\leq 0  
\end{align*}
 on $C$ in the sense of distributions, where the last inequality follows since $\V$ is non-decreasing in each argument and $\delta>0$ is defined in \eqref{eq:definitionofdeltafinitevariation}. Then by continuity of the functions $\V$ and $\G$ (recall that $\G$ is at least continuous on $(0,\infty)\times (0,\infty)$ and right-continuous at points of the form $(t,0)$ for $t\geq 0$) we can derive
\begin{align}
\label{eq:VPidy+Gleq0}
\int_{(-\infty,0)} [\V(t,y)-\V(t,0)] \Pi(\dd y)+\G(t,0) \leq 0
\end{align} 
for all $t\in [0,t_b)$. \\

Next, we show that the set $\{t\in [0,\m): \b(t)=0 \}$ is non empty. We proceed by contradiction, assume that $\b(t)>0$ for all $t\in [0,\m)$ so that $t_b=\m$. Taking $t\uparrow \m$ in \eqref{eq:VPidy+Gleq0} and applying dominated convergence theorem, we obtain that 
\begin{align*}
0\geq \lim_{t\uparrow \m }\left\{ \int_{(-\infty,0)} [\V(t,y)-\V(t,0)] \Pi(\dd y)+\G(t,0) \right\}=\G(\m,0)=F^{(\theta)}(0)>0,
\end{align*}

where the strict inequality follows from $F^{(\theta)}(0)=\frac{\theta}{\Phi(\theta)} W^{(\theta)}(0)=\frac{\theta}{\delta \Phi(\theta)}>0$ since $X$ is of finite variation. Therefore, we observe a  contradiction which shows that $\{t\in [0,\m): \b(t)=0 \}\neq \emptyset$. Moreover, by the definition, we have that $t_b=\inf\{ t\in [0,\m): \b(t)=0\}$. \\

Next we find an expression for $\V(t,x)$ when $t\in (0,\m)$ and $x< 0$. Since $\b(t)\geq 0$ for all $t\in [0,\m)$, we have that  
\begin{align}
\label{eq:Vforxnegative}
\V(t,x)&=\E_x\left(\int_0^{\tau_{0}^+ \wedge (\m-t)} (1-2e^{-\theta (t+s)} ) \dd s \right)+\E_x\left(\I_{\{\tau_{0}^+<\m-t \}} \V(t+\tau_0^+,0) \right)\nonumber\\
&=\V_B(t,x)+\E_x\left(\I_{\{\tau_{0}^+<\m-t \}} \V(t+\tau_0^+,0) \right) ,
\end{align}
where the first equality follows since $X_s\leq 0$ for all $s\leq \tau_0^+$ and $G(t,x)=1-2e^{-\theta t}$ for all $x<0$. Hence, in particular, we have that $\V(t,x)=\V_B(t,x)$ for all $t\in [t_b,\m)$ and $x\in \R$.\\

We show that \eqref{eq:characterisationoftb} holds. From the discussion after Lemma \ref{lemma:variationalcharacterization}, we know that
\begin{align*}
\varphi^{(\theta)}(t,x)=\int_{(-\infty,0)} \V(t,x+y)\Pi (\dd y)+\G(t,x) > 0 
\end{align*}
for all $x>0$ and $t\geq t_b$. Then by taking $x\downarrow 0$, making use of the right continuity of $x\mapsto G(t,x)$, continuity of $\V$ (see Lemma \ref{lemma:Vcontinuity}) and applying dominated convergence theorem, we derive that
\begin{align*}
\int_{(-\infty,0)} \V(t_b,y)\Pi (\dd y)+\G(t_b,0) \geq 0.
\end{align*}
In particular, if $t_b=0$, \eqref{eq:characterisationoftb} holds since $t\mapsto \V(t,y)$ (for all $y\in \R$) and $\G(t,0)$ are non-decreasing functions. If $t_b>0$, taking $t\uparrow t_b$ in \eqref{eq:VPidy+Gleq0} gives us
\begin{align*}
\int_{(-\infty,0)} \V(t_b,y)\Pi(\dd y)+\G(t,0)\leq 0.
\end{align*}
Hence, we have that $\int_{(-\infty,0)}\V_B(t_b,y)\Pi(\dd y)+\G(t_b,0)=0$ with \eqref{eq:characterisationoftb} becoming clear due to the fact that $t\mapsto \V_B(t,x)$ is non-decreasing. If $X$ is a process of infinite variation, we have that $\hth(t)>0$ for all $t\in [0,\m)$ and therefore $\G(t,x)< 0$ for all $t\in [0,\m)$ and $x\leq 0$ which implies that 
\begin{align*}
t_b=\m= \inf \left\{ t\in [0,\m]:\int_{(-\infty,0)}\V_B(t,y)\Pi(\dd y) +\G(t,0)\geq 0 \right\}. 
\end{align*}
\end{proof}

Now we prove that the derivatives of $\V$ exist at the boundary $\b$ for those points in which $\b$ is strictly positive.

\begin{lemma}
For all $t\in [0,t_b)$, the partial derivatives of $\V(t,x)$ at the point $(t,\b(t))$ exist and are equal to zero, i.e.,
\begin{align*}
\frac{\partial}{\partial t} \V(t,\b(t))=0 \qquad \text{and} \qquad \frac{\partial}{\partial x} \V(t,\b(t))=0.
\end{align*}

\end{lemma}

\begin{proof}
First, we prove that the assertion in the first argument. Using a similar idea as in Lemma \ref{lemma:Vcontinuity}, we have that for any $t<t_b$, $x\in \R$ and $h>0$,
\begin{align*}
0\leq  \frac{ \V(t,\b(t))- \V(t-h,\b(t)) }{h}& \leq  2 \E_{\b(t)}\left( \int_{0}^{\tau^*_{h} } \frac{[e^{-\theta(r+t-h)}-e^{-\theta(r+t)}]}{h}  \dd r\right)\\
& \leq 2 \E_{\b(t)}\left( \int_{0}^{\tau^+_{\b(t-h)} } \frac{[e^{-\theta(r+t-h)}-e^{-\theta(r+t)}]}{h}  \dd r\right)\\
&= \frac{[e^{-\theta(t-h)}-e^{-\theta t}]}{h} \frac{1}{\theta} \E_{\b(t)}\left( 1-e^{-\theta \tau^+_{\b(t-h)}}  \right)
\end{align*}
where $\tau^*_h=\inf\{r \in [0,\m-t+h]: X_r  \geq \b(r+t-h) \}$ is the optimal stopping time for $\V(t-h,x)$ and the second inequality follows since $b$ is non increasing. Hence, by \eqref{eq:laplacetransformtau0+} and continuity of $\b$ we obtain that
\begin{align*}
 \lim_{h\downarrow 0}  \frac{ \V(t,\b(t))- \V(t-h,\b(t)) }{h}= 0.
\end{align*}
Now we show that the partial derivative of the second argument exists at $\b(t)$ and is equal to zero. Fix any time $t \in [0,t_b)$, $\varepsilon>0$ and $x\leq b^{(\theta)}(t)$ (without loss of generality, we assume that $\varepsilon<x$).  By a similar argument that as provided in Lemma \ref{lemma:Vcontinuity}, we obtain that 

\begin{align*}
V^{(\theta)}&(t,x)-V^{(\theta)}(t,x-\varepsilon)\\
& \leq 2   \int_{-\infty}^{\b(t)} [F^{(\theta)}(z+\varepsilon)-F^{(\theta)}(z)] \left[e^{-\Phi(\theta)(\b(t)-x+\varepsilon)} W^{(\theta)}(\b(t)-z)-W^{(\theta)}(x-\varepsilon-z) \right]\dd z\\
& =  2 e^{-\Phi(\theta)(\b(t)-x+\varepsilon)}  \int_{x-\varepsilon}^{\b(t)} [F^{(\theta)}(z+\varepsilon)-F^{(\theta)}(z)]  W^{(\theta)}(\b(t)-z)\dd z\\
&\qquad + 2  \int_{0}^{x-\varepsilon} [F^{(\theta)}(z+\varepsilon)-F^{(\theta)}(z)] \left[e^{-\Phi(\theta)(\b(t)-x+\varepsilon)} W^{(\theta)}(\b(t)-z)-W^{(\theta)}(x-\varepsilon-z) \right]\dd z\\
&\qquad +2  \int_{-\varepsilon}^{0} F^{(\theta)}(z+\varepsilon) \left[e^{-\Phi(\theta)(\b(t)-x+\varepsilon)} W^{(\theta)}(\b(t)-z)-W^{(\theta)}(x-\varepsilon-z) \right]\dd z .
\end{align*}
Dividing by $\varepsilon$, we have that for $t\in [0,\m)$ and $\varepsilon<x $ that 
\begin{align*}
0\leq \frac{  V^{(\theta)}(t,x)-V^{(\theta)}(t,x-\varepsilon)}{\varepsilon}\leq R_1^{(\varepsilon)}(t,x)+R_2^{(\varepsilon)}(t,x)+R_3^{(\varepsilon)}(t,x),
\end{align*}
where 
\begin{align*}
 R_1^{(\varepsilon)}(t,x)&=   2 e^{-\Phi(\theta)(\b(t)-x+\varepsilon)} \frac{1}{\varepsilon} \int_{x-\varepsilon}^{\b(t)} [F^{(\theta)}(z+\varepsilon)-F^{(\theta)}(z)]  W^{(\theta)}(\b(t)-z)\dd z,\\
R_2^{(\varepsilon)}(t,x)&=  2\frac{1}{\varepsilon}  \int_{0}^{x-\varepsilon} [F^{(\theta)}(z+\varepsilon)-F^{(\theta)}(z)] \left[e^{-\Phi(\theta)(\b(t)-x+\varepsilon)} W^{(\theta)}(\b(t)-z)-W^{(\theta)}(x-\varepsilon-z) \right]\dd z,\\
R_3^{(\varepsilon)}(t,x)&= 2  \frac{1}{\varepsilon}\int_{-\varepsilon}^{0} F^{(\theta)}(z+\varepsilon) \left[e^{-\Phi(\theta)(\b(t)-x+\varepsilon)} W^{(\theta)}(\b(t)-z)-W^{(\theta)}(x-\varepsilon-z) \right]\dd z .
\end{align*}
By using that $W$ and $F$ are non-decreasing, that $W^{(\theta)}$ (and hence $F^{(\theta)}$) has left and right derivatives and the dominated convergence theorem we can show that for $t\in [0,t_b)$, $\lim_{\varepsilon \downarrow 0} R_i^{(\varepsilon)}(t,\b(t))=0$ for each $i=1,2, 3$. Hence, we have that
\begin{align*}
\lim_{\varepsilon \downarrow 0} \frac{V^{(\theta)}(t,\b(t))-V^{(\theta)}(t,\b(t)-\varepsilon)}{\varepsilon} = 0
\end{align*}
proving that $x \mapsto V^{(\theta)}(x,t)$ is differentiable at $b^{(\theta)}(t)$ with $ \partial/\partial x V^{(\theta)}(t,b^{(\theta)}(t))=0$ for $t\in [0,t_b)$.

\end{proof}
The next theorem looks at how the value function $\V$ and the curve $\b$ can be characterised as a solution of non-linear integral equations within a certain family of functions. These equations are in fact generalisations of the free boundary equation (see e.g. \cite{peskir2006optimal} Section 14.1 in a diffusion setting) in the presence of jumps. It is important to mention that the proof of Theorem \ref{thm:Vsatisfiesequation} is mainly inspired by the ideas of \cite{du2008predicting} with some extensions to allow for the presence of jumps. 

\begin{thm}
\label{thm:Vsatisfiesequation}
Let $X$ be a spectrally negative L\'evy process and let $t_b$ be as characterised in \eqref{eq:characterisationoftb}. For all $t\in [0,t_b)$ and $x\in \R$, we have that 

\begin{align}
\label{eq:characterizationofV}
\V(t,x)&=\E_x\left( \int_0^{\m-t	} \G(r+t,X_r)\I_{\{X_r <\b(r+t) \}}\dd r \right)\nonumber \\
&\qquad-\E_x\left( \int_0^{\m-t} \int_{(-\infty,\b(r+t)-X_r)} \V(r+t,X_r+y)\Pi(\dd y) \I_{\{X_r> \b(r+t) \}}\dd r \right)
\end{align}
and $\b(t)$ solves the equation 
\begin{align}
\label{eq:characterizationofb}
\E_{\b(t)}&\left( \int_0^{\m-t	} \G(r+t,X_r)\I_{\{X_r <\b(r+t) \}}\dd r \right)\nonumber \\
&\qquad-\E_{\b(t)}\left( \int_0^{\m-t} \int_{(-\infty,\b(r+t)-X_r)} \V(r+t,X_r+y)\Pi(\dd y) \I_{\{X_r> \b(r+t) \}}\dd r \right)=0.
\end{align}
If $t\in [t_b,\m)$, we have that $\b(t)=0$ and 
\begin{align}
\label{eq:Vaftertb}
\V(t,x)= \E_x(\tau_0^+\wedge (\m-t))-\frac{2}{\theta}e^{-\theta t}[1-\E_x(e^{-\theta (\tau_0^+ \wedge (\m-t))})] 
\end{align} 
for all $x\in \R$. Moreover, the pair $(\V,\b)$ is uniquely characterised as the solutions to equations \eqref{eq:characterizationofV}-\eqref{eq:Vaftertb} in the class of continuous functions in $\R_+\times\R$ and $\R_+$, respectively, such that $\b\geq \hth$, $\V\leq 0$ and $\int_{(-\infty,0)} \V(t,x+y)\Pi(\dd y) +\G(t,x) \geq 0$ for all $t\in [0,t_b)$ and $x\geq \b(t)$.
\end{thm}

\subsection{Proof of Theorem \ref{thm:Vsatisfiesequation}}
\label{sec:proof}

Since the proof of Theorem \ref{thm:Vsatisfiesequation} is rather long, we split it into a series of Lemmas. This subsection is entirely dedicated for this purpose. With the help of It\^o formula and following an analogous argument as in \cite{Lamberton2013} (in the infinite variation case), we prove that $\V$ and $\b$ are solutions to the integral equations listed above. The finite variation case is proved using an argument that considers the consecutive times in which $X$ hits the curve $\b$.

\begin{lemma}
The pair $(\V,\b)$ are solutions to the equations \eqref{eq:characterizationofV}-\eqref{eq:Vaftertb}.
\end{lemma}

\begin{proof}
Recall from Lemma \ref{lemma:characterizationoftb} that, when $t_b<\m$, the value function $\V(t,x)$ satisfies equation \eqref{eq:Vaftertb} for $t\in [t_b,\m)$ and $x\in \R$. We also have that equation \eqref{eq:characterizationofb} follows from \eqref{eq:characterizationofV} by letting $x=\b(t)$ and using that $\V(t,\b(t))=0$.\\

We proceed to show that $(\V,\b)$ solves equation \eqref{eq:characterizationofV}. First, we assume that $X$ is a process of infinite variation. We follow an analogous argument as \cite{Lamberton2013} (see Theorem 3.2). Consider a regularized sequence $\{\rho_n \}_{n\geq 1}$ of non-negative $C^{\infty}(\R_+\times \R)$ functions with support in $[-1/n,0]\times [-1/n,0]$ such that $\int_{-\infty}^0 \int_{-\infty}^0 \rho_n(s,y)\dd s \dd y=1$. For every $n\geq 1$, define the function $\V_n$ by 

\begin{align*}
\V_n(t,x)=(\V \ast \rho_n)(t,x)=\int_{-\infty}^0 \int_{-\infty}^0 \V(t+s,x+y)\rho_n(s,y) \dd s \dd y.
\end{align*}   
for any $(t,x)\in [1/n,\infty)\times \R$. Then for each $n\geq 1$, the function $\V_n$ is a $C^{1,2}(\R_+\times \R)$ bounded function (since $\V$ is bounded). Moreover, it can be shown that $ \V_n \uparrow V$ on $\R_+\times \R$ when $n\rightarrow \infty$ and that (see \cite{lamberton2008critical}, proof of Proposition 2.5), 
\begin{align}
\label{eq:generatorappliedtoVrhon}
\frac{\partial}{\partial t}\V_n(t,x)+\mathcal{A}_{X}(\V_n)(t,x) =-(\G*\rho_n)(u,x) \qquad \text{for all } (t,x) \in [1/n,\infty)\times\R \cap C,
\end{align}
where $\mathcal{A}_{X}$ is the infinitesimal generator of $X$ given in \eqref{eq:generatorofX} and $C=\R_+\times \R \setminus D$. Let $t\in (0,t_b]$, $m>0$ such that $t>1/m$ and $x\in \R$. Applying It\^o formula to $\V_n(t+s,X_{s}+x )$, for $s\in[0,\m-t]$, we obtain that for any $n\geq m$,
\begin{align*}
\V_n(s+t, X_{s}+x) 
&= \V_n(t,x)+ M_{s}^{t,n}+\int_0^{s } \left[ \frac{\partial}{\partial t} \V_n(r+t,X_r+x) + \mathcal{A}_{X}( \V_n)(r+t,X_r+x)\right]\dd r,
\end{align*}
where $\{ M_{s}^{t,n}, t\geq 0 \}$ is a zero mean martingale.
Hence, taking expectation and using \eqref{eq:generatorappliedtoVrhon}, we derive that
\begin{align*}
\E(\V_n(s+t, X_{s}+x)) 
&= \V_n(t,x)+\E\left(\int_0^{s } \left[ \frac{\partial}{\partial t} \V_n(r+t,X_r+x) + \mathcal{A}_{X}( \V_n)(r+t,X_r+x)\right]\dd r \right)\\
&= \V_n(t,x)-\E\left(\int_0^{s } (\G*\rho_n)(r+t,X_r+x)\I_{\{X_r <\b(r+t) \}}\dd r \right)\\
&\qquad + \E\left(\int_0^{s }  \int_{(-\infty,0)} \V_n (r+t,X_r+x+y) \Pi(\dd y)\I_{\{X_r >\b(r+t) \}}\dd r \right),
\end{align*}
where we used the fact that $\b(s)$ is finite for all $s\geq 0$ and that $\P(X_s+x=b(t+s))=0$ for all $s>0$ and $x\in \R$ when $X$ is of infinite variation (see \cite{sato1999levy}). Taking $s=\m-t$, using the fact that $\V(\m,x)=0$ for all $x\in \R$ and letting $n\rightarrow \infty$ (by the dominated convergence theorem), we obtain that \eqref{eq:characterizationofV} holds for any $(t,x)\in (0,t_b)\times\R$. The case when $t=0$ follows by continuity.\\

For the finite variation case, we define the auxiliary function
\begin{align*}
R^{(\theta)}(t,x)&=\E_x\left( \int_0^{\m-t	} \G(r+t,X_r)\I_{\{X_r <\b(r+t) \}}\dd r \right)\nonumber \\
&\qquad-\E_x\left( \int_0^{\m-t} \int_{(-\infty,0)} \V(r+t,X_r+y)\Pi(\dd y) \I_{\{X_r>\b(r+t) \}}\dd r \right)
\end{align*}
for all $(t,x)\in \R_+\times \R$. We then prove that $R^{(\theta)}=\V$. First, note that from the discussion after Lemma \ref{lemma:variationalcharacterization} we have that $\int_{(-\infty,0)} \V(t,x+y) +\G(t,x)\geq 0$ for all $(t,x)\in D$. Then we have that for all $(t,x)\in [0,\m]\times \R$,
\begin{align*}
|R^{(\theta)}(t,x)|  
\leq \E_x\left( \int_0^{\m-t	} |\G(r+t,X_r)| \dd r \right)
 \leq \m-t,
\end{align*}
where we used that $|\G|\leq 1$ in the last inequality. For each $(t,x)\in \R_+\times \R$, we define the times at which the process $X$ hits the curve $\b$. Let $\tau_b^{(1)}=\inf\{s\in [0,\m-t]: X_s\geq \b(s+t) \}$ and for $k\geq 1$,
\begin{align*}
\sigma_b^{(k)}&=\inf\{s\in [\tau_b^{k},\m-t] : X_s<\b(s+t)  \}\\
\tau_b^{(k+1)}&=\inf\{s\in [\sigma_b^{k},\m-t] : X_s \geq \b(s+t)  \},
\end{align*}
where in this context, we understand that $\inf \emptyset =\m-t$. Taking $t\in [0,\m]$ and $x>\b(t)$ and gives us  

\begin{align*}
R^{(\theta)}(t,x)&=-\E_x\left( \int_0^{\sigma_b^{(1)}} \int_{(-\infty,0)} \V(r+t,X_r+y)\Pi(\dd y) \dd r \right)+\E_x\left(\int_{\sigma_b^{(1)}}^{\tau_b^{(2)}	} \G(r+t,X_r)\dd r \right)\\
&\qquad+\E_x\left(\I_{\{\tau_b^{(2)}<\m-t \}} \int_{\tau_b^{(2)}}^{\m-t	} \G(r+t,X_r)\I_{\{X_r <\b(r+t) \}}\dd r \right) \\
&\qquad-\E_x\left( \I_{\{\tau_b^{(2)}<\m-t \}} \int_{\tau_b^{(2)}}^{\m-t} \int_{(-\infty,0)} \V(r+t,X_r+y)\Pi(\dd y) \I_{\{X_r> \b(r+t) \}}\dd r \right)\\
&=-\E_x\left( \int_0^{\sigma_b^{(1)}} \int_{(-\infty,0)} \V(r+t,X_r+y)\Pi(\dd y) \dd r \right)+\E_x( \V(t+\sigma_b^{(1)} , X_{\sigma_b^{(1)} })\I_{\{ \sigma_b^{(1)}<\m-t \}} )\\
&\qquad+\E_x(R^{(\theta)}(t+\tau_{b}^{(2)}, X_{\tau_{b}^{(2)}})\I_{\{ \tau_b^{(2)}<\m-t \}}  ),
\end{align*}
where the last equality follows from the strong Markov property applied at time $\sigma_b^{(1)}$ and $\tau_b^{(2)}$, respectively, and the fact that $\tau_D$ is optimal for $\V$. Using the compensation formula for Poisson random measures (see \cite{kyprianou2014fluctuations} Theorem 4.4), it can be shown that 
\begin{align*}
\E_x\left( \int_0^{\sigma_b^{(1)}} \int_{(-\infty,0)} \V(r+t,X_r+y)\Pi(\dd y) \dd r \right)=\E_x(\V(t+\sigma_b^{(1)},X_{\sigma_b^{(1)}} ) \I_{\{ \sigma_b^{(1)}<\m-t \}}).
\end{align*}
Hence, for all $(t,x)\in D$, we have that
\begin{align*}
R^{(\theta)}(t,x)=\E_x(R^{(\theta)}(t+\tau_{b}^{(2)}, X_{\tau_{b}^{(2)}})\I_{\{ \tau_b^{(2)}<\m-t \}}  ).
\end{align*}
Using an induction argument, it can be shown that for all $(t,x)\in D$ and $n\geq 2$,
\begin{align}
\label{eq:functionRintermsoftaubn}
R^{(\theta)}(t,x)=\E_x(R^{(\theta)}(t+\tau_{b}^{(n)}, X_{\tau_{b}^{(n)}})\I_{\{ \tau_b^{(n)}<\m-t \}}  )=\E_x(R^{(\theta)}(t+\tau_{b}^{(n)}, X_{\tau_{b}^{(n)}}) ),
\end{align}
where the last equality follows since $R^{(\theta)}(\m,x)=0$ for all $x\in \R$. Since $X$ is of finite variation, it can be shown that for all $x\in \R$, $\lim_{n \rightarrow \infty} \tau_b^{(n)}=\m-t$ $\P_x$-a.s. Therefore, from \eqref{eq:functionRintermsoftaubn} and taking $n\rightarrow \infty$, we conclude that for all $(t,x)\in D$,
\begin{align*}
|R^{(\theta)}(t,x)|\leq \lim_{n\rightarrow \infty}\E_x\left(|R^{(\theta)}(t+\tau_{b}^{(n)}, X_{\tau_{b}^{(n)}})| \right) \leq \lim_{n\rightarrow \infty} \E_x(\m-t-\tau_b^{(n)})=0,
\end{align*}
where the last inequality follows from the dominated convergence theorem. On the other hand, if we take $t\in [0,\m]$ and $x<\b(t)$ we have, by the strong Markov property applied to the filtration at time $\tau_b^{(1)}$, that
\begin{align*}
R^{(\theta)}(t,x)=\E_x\left( \int_0^{\tau_b^{(1)}} \G(r+t,X_r) \dd r \right) +\E_x(R^{(\theta)} (t+\tau_b^{(1)} , X_{\tau_b^{(1)}}  ))=\V(t,x),
\end{align*}
where we used the fact that $\tau_b^{(1)}$ is an optimal stopping time for $\V$ and that $R^{(\theta)}$ vanishes on $D$. So then \eqref{eq:characterizationofV} also holds in the finite variation case. 

\end{proof}
Next we proceed to show the uniqueness result. Suppose that there exist a non-positive continuous function $\U: [0,\m]\times \R \mapsto (-\infty,0]$ and a continuous function $\ct$ on $[0,\m)$ such that $\ct\geq \hth$ and $\ct(t)=0$ for all $t\in [t_b,\m)$. We assume that the pair $(\U,\ct)$ solves the equations

\begin{align}
\label{eq:unicityofV}
\U(t,x)&=\E_x\left( \int_0^{\m-t} \G(r+t,X_r)\I_{\{X_r< \ct (r+t) \}}\dd r\right)\nonumber\\
&\qquad -\E_x\left(\int_0^{\m-t}\int_{(-\infty,\ct(r+t)-X_r)}\U (r+t,X_r+y) \Pi(\dd y) \I_{\{X_r> \ct (r+t) \}} \dd r\right)
\end{align}
and 

\begin{align}
\label{eq:unicityofb}
\E_{\ct(t)}&\left( \int_0^{\m-t} \G(r+t,X_r)\I_{\{X_r< \ct (r+t) \}}\dd r\right)\nonumber\\
&\qquad -\E_{\ct(t)}\left(\int_0^{\m-t}\int_{(-\infty,\ct(r+t)-X_r)}\U (r+t,X_r+y) \Pi(\dd y) \I_{\{X_r > \ct (r+t) \}} \dd r\right)=0
\end{align}
when $t\in [0,t_b)$ and $x\in \R$. For $t\in [t_b,\m)$ and $x\in \R$, we assume that 
\begin{align}
\label{eq:Uaftertb}
\U(t,x)= \E_x(\tau_0^+\wedge (\m-t))-\frac{2}{\theta}e^{-\theta t}[1-\E_x(e^{-\theta (\tau_0^+ \wedge (\m-t))})].
\end{align}
In addition, we assume that 
\begin{align}
\label{eq:intVpidyGpositiveonD}
\int_{(-\infty,\ct(t)-x)} \U(t,x+y)\Pi(\dd y)+\G(t,x)\geq 0 \qquad \text{for all } t\in [0,t_b) \text{ and } x>\ct(t).
\end{align}
Note that $(\U,\ct)$ solving the above equations means that $\U(t,\ct(t))=0$ for all $t \in [0,\m)$ and $\U(\m,x)=0$ for all $x\in \R$. Denote $D_c$ as the ``stopping region'' under the curve $\ct$, i.e., $D_c=\{(t,x) \in  [0,\m]\times \R: x \geq \ct(t) \}$ and recall that $D=\{(t,x) \in  [0,\m]\times \R: x \geq \b(t) \}$ is the ``stopping region'' under the curve $\b$. We show that $\U$ vanishes on $D_c$ in the next Lemma.

\begin{lemma}
\label{lemma:UvanishesatDc}
We have that $\U(t,x)=0$ for all $(t,x)\in D_c$.
\end{lemma}

\begin{proof}
Since the statement is clear for $(t,x)\in [t_b,\m)\times [0,\infty)$, we take $t\in [0,t_b)$ and $x\geq \ct(t)$. Define $\sigma_c$ to be the first time that the process is outside $D_c$ before time $\m-t$, i.e.,
\begin{align*}
\sigma_c=\inf\{0\leq s\leq \m-t: X_{s} < \ct(t+s) \},
\end{align*}
where in this context, we understand that $\inf \emptyset= \m-t$. From the fact that $X_{r}\geq \ct(t+r)$ for all $r< \sigma_c$ and the strong Markov property at time $\sigma_c$, we obtain that
\begin{align*}
\U(t,x)
&=\E_x(\U(t+\sigma_c,X_{\sigma_c})) -\E_x\left(\int_0^{\sigma_c}\int_{(-\infty,\ct(r+t)-X_r)}\U (r+t,X_r+y) \Pi(\dd y) \dd r\right)\\
&=\E_x(\U(t+\sigma_c,X_{\sigma_c})\I_{\{\sigma_c<\m-t, X_{\sigma_c}<\ct (t+\sigma_c) \}} ) \\
&\qquad-\E_x\left(\int_0^{\sigma_c}\int_{(-\infty,\ct(r+t)-X_r)}\U (r+t,X_r+y) \Pi(\dd y) \dd r\right),
\end{align*}
where the last equality follows since $\U(\m,x)=0$ for all $x \in \R$ and $\U(t,\ct(t))=0$ for all $t\in [0,t_b)$.
Then, applying the compensation formula for Poisson random measures (see \cite{kyprianou2014fluctuations} Theorem 4.4) we get 

\begin{align*}
\E_x(\U(t+\sigma_c,X_{\sigma_c})\I_{\{\sigma_c<\m-t, X_{\sigma_c}<\ct (t+\sigma_c) \}} )
&=\E_{x}\left( \int_0^{\sigma_c}\int_{(-\infty,\ct(t+r)-X_{t+r})}\U(t+r,X_{r}+y) \Pi(\dd y) \dd r\right).
\end{align*}
Hence  $\U(t,x)=0$ for all $(t,x)\in D_c$ as we claimed.
\end{proof}

The next lemma shows that $\U$ can be expressed as an integral involving only the gain function $\G$ stopped at the first time the process enters the set $D_c$. As a consequence, $\U$ dominates the function $\V$. 

\begin{lemma}
\label{lemma:VboundsUbybelow}
We have that $\U(t,x)\geq \V(t,x)$ for all $(x,t)\in \R\times [0,\m]$,

\end{lemma}

\begin{proof}
Note that we can assume that $t\in [0,t_b)$ because for $(t,x)\in D_c$, we have that $\U(t,x)=0\geq \V(t,x)$ and for $t\in [t_b,\m)$, $\U(t,x)=\V(t,x)$ for all $x\in \R$. Consider the stopping time 
\begin{align*}
\tau_c=\inf\{ s \in [0,\m-t] : X_{s} \geq  \ct(t+s) \}.
\end{align*}
Let $x\leq \ct (t)$, using the fact that $X_{r}< \ct(t+r)$ for all $r\leq \tau_c$ and the strong Markov property at time $\tau_c$, we obtain that
\begin{align}
\label{eq:representationofUbelowc}
\U(t,x)=\E_x\left( \int_0^{\tau_c} \G(r+t,X_r)\dd r\right) +\E_x(\U(t+\tau_c,X_{\tau_c}))=\E_x\left( \int_0^{\tau_c} \G(r+t,X_r)\dd r\right),
\end{align}
where the second equality follows since $X$ creeps upwards and therefore $X_{\tau_c}=\ct(t+\tau_c)$ for $\{\tau_c<\m-t \}$ and $\U(\m,x)=0$ for all $x\in \R$. Then from the definition of $\V$ (see \eqref{eq:optimalstoppingforallx}), we have that 
\begin{align*}
\U(t,x)\geq \inf_{\tau \in \mathcal{T}} \E_{t,x}\left( \int_0^{\tau} G^{(\theta)}(X_{t+r},t+r) dr\right)=\V(t,x).
\end{align*}
Therefore  $\U \geq \V$ on $ [0,\m]\times \R$.
\end{proof}
We proceed by showing that the function $\ct$ is dominated by $\b$. In the upcoming lemmas, we show that equality indeed holds. 
\begin{lemma}
We have that $\b(t)\geq \ct(t)$ for all $t\in [0,\m)$.
\end{lemma}

\begin{proof}
The statement is clear for $t\in [t_b,\m)$. We prove the statement by contradiction. Suppose that there exists a value $t_0\in [0,t_b)$ such that $\b(t_0)<\ct(t_0)$ and take $x\in (\b(t_0),\ct(t_0))$. Consider the stopping time
\begin{align*}
\sigma_b=\inf\{s\in [0,\m-t_0]: X_{s} < \b(t_0+s) \}.
\end{align*}
Applying the strong Markov property to the filtration at time $\sigma_b$, we obtain that 
\begin{align*}
\U(t_0,x)&=
\E_{x	}(\U(t_0+\sigma_b,	X_{\sigma_b}))
+\E_{x}\left(\int_0^{\sigma_b}\G(t_0+r,X_{r})\I_{\{ X_{r} <\ct(t_0+r)   \}}\dd r\right)\\
&\qquad-\E_{x}\left(\int_0^{\sigma_b} \int_{(-\infty,0)} \U(t+r,X_{r}+y)\Pi( \dd y) \I_{\{X_{r} > \ct(t_0+r) \}}\dd r\right),
\end{align*}
where we used the fact that $\U(t,x)=0$ for all $(t,x)\in D_c$. From Lemma \ref{lemma:VboundsUbybelow} and the fact that $\U\leq 0$ (by assumption), we have that for all $t\in [0,\m)$ and $x>\b(t)$, $ \U(t,x)=0$. Hence, by the compensation formula for Poisson random measures, we obtain that 
\begin{align*}
0 &\geq \U(t_0,x)\\
&=
\E_{x	}(\U(t_0+\sigma_b,	X_{\sigma_b})\I_{\{\sigma_b<\m-t, X_{\sigma_b}<\b(t_0+\sigma_b) \}})
+\E_{x}\left(\int_0^{\sigma_b}\G(t_0+r,X_{r})\I_{\{ X_{r} <\ct(t_0+r)   \}}\dd r\right)\\
&\qquad-\E_{x}\left(\int_0^{\sigma_b} \int_{(-\infty,0)} \U(t+r,X_{r}+y)\Pi( \dd y) \I_{\{X_{r} >\ct(t_0+r) \}}\dd r\right)\\
&=\E_{x}\left( \int_0^{\sigma_b}\int_{(-\infty,0)} \U(t_0+r,X_{r}+y) \Pi(\dd y)\dd r\right)+\E_{x}\left(\int_0^{\sigma_b}\G(t_0+r,X_{r})\I_{\{ X_{r} <\ct(t_0+r)   \}}\dd r\right)\\
&\qquad-\E_{x}\left(\int_0^{\sigma_b} \int_{(-\infty,0)} \U(t+r,X_{r}+y)\Pi( \dd y) \I_{\{X_{r} > \ct(t_0+r) \}}\dd r\right)\\
&=\E_{x}\left( \int_0^{\sigma_b} \left[ \int_{(-\infty,0)} \U(t_0+r,X_{r}+y) \Pi(\dd y)+\G(t_0+r,X_{r}) \right]\I_{\{ X_{r} <\ct(t_0+r)   \}}\dd r\right).
\end{align*}
Recall from the discussion after Lemma \ref{lemma:variationalcharacterization} that the function $\varphi_t^{(\theta)}$ is strictly positive on $D$. Hence, we obtain that for all $(t,x)\in D$,
 
\begin{align*}
 \int_{(-\infty,0)} \U(t,x+y)\Pi(\dd y)+\G(t,x)\geq \int_{(-\infty,0)} \V(t,x+y)\Pi(\dd y)+\G(t,x)=\varphi_t^{(\theta)}(t,x)>0.
\end{align*}
The assumption that $\b(t_0)<\ct(t_0)$ together with the continuity of the functions $\b$ and $\ct$ mean that there exists $s_0\in(t_0,\m)$ such that $\b(r)<\ct(r)$ for all $r\in [t_0,s_0]$. Consequently, the $\P_{x}$ probability of $X$ spending a strictly positive amount of time (with respect to Lebesgue measure) in this region is strictly positive. We can then conclude that 
\begin{align*}
0 \geq \E_{x}\left( \int_0^{\sigma_b} \left[ \int_{(-\infty,0)} \U(t_0+r,X_{r}+y) \Pi(\dd y)+\G(t_0+r,X_{r}) \right]\I_{\{ X_{r} <\ct(t_0+r)   \}}\dd r\right)>0.
\end{align*}
This is a contradiction and therefore we conclude that $\b(t) \geq \ct(t)$ for all $t\in [0,\m)$.
\end{proof}

%
%
Note that the definition of $\U$ on $[t_b,\m)\times \R$ (see equation \eqref{eq:Uaftertb}) together with condition \eqref{eq:intVpidyGpositiveonD} imply that 
\begin{align*}
\int_{(-\infty,0)} \U(t,x+y)\Pi(\dd y)+ \G(t,x)\geq 0 
\end{align*}
for all $t\in [0,\m)$ and $x>\ct(t)$. The next Lemma shows that $\U$ and $\V$ coincide.

\begin{lemma}
We have that $\b(t)=\ct(t)$ for all $t\geq 0$ and hence $\V=\U$.
\end{lemma}

\begin{proof}
We prove that $\b=\ct$ by contradiction. Assume that there exists $s_0$ such that $\b(s_0)>\ct(s_0)$. Since $\ct(t)=\b(t)=0$ for all $t\in [t_b,\m)$, we deduce that $s_0\in [0,t_b)$. Let $\tau_b$ be the stopping time
\begin{align*}
\tau_b=\inf\{t\geq 0: X_s \geq \b(s_0+t) \}.
\end{align*}
With the Markov property applied to the filtration at time $\tau_b$, we obtain that for any $x\in (\ct(s_0),\b(s_0))$

\begin{align*}
\E_{x}(\U(s_0+\tau_b, X_{\tau_b}	))
&=\U(s_0,x)
-\E_{x}\left(\int_0^{\tau_b}\G(r+s_0,X_r)\I_{\{ X_{r} <\ct(r+s_0) \}} \dd r\right)\\
&\qquad+\E_{x}\left(\int_0^{\tau_b} \int_{(-\infty,0)} \U(r+s_0, X_r+y)\I_{\{ X_{r} >\ct(r+s_0) \}}\Pi( \dd y)\dd r\right)\\
&\geq \V(s_0,x)
-\E_{x}\left(\int_0^{\tau_b}\G(r+s_0,X_r)\I_{\{ X_{r} <\ct(r+s_0) \}} \dd r\right)\\
&\qquad+\E_{x}\left(\int_0^{\tau_b} \int_{(-\infty,0)} \U(r+s_0, X_r+y)\I_{\{ X_{r} > \ct(r+s_0) \}}\Pi( \dd y)\dd r\right)\\
&=\E_{x}\left(\int_0^{\tau_b}\G(r+s_0,X_r)\I_{\{ X_{r} \geq \ct(r+s_0) \}} \dd r\right)\\
&\qquad+\E_{x}\left(\int_0^{\tau_b} \int_{(-\infty,0)} \U(r+s_0, X_r+y)\I_{\{ X_{r} > \ct(r+s_0) \}}\Pi( \dd y)\dd r\right),
\end{align*}
where the second inequality follows from the fact that $\U\geq \V$ (see Lemma \ref{lemma:VboundsUbybelow}) and the last equality follows as $\tau_b$ is the optimal stopping time for $\V(s_0,x)$. Note that since $X$ creeps upwards, we have that $\U(s_0+\tau_b,X_{\tau_b})=\U(s_0+\tau_b,\b(s_0+\tau_b))=0$. Hence,  

 \begin{align*}
 \E_{x}&\left(\int_0^{\tau_b}\G(r+s_0,X_r)\I_{\{ X_{r} \geq \ct(r+s_0) \}} \dd r\right)\\
&\qquad+\E_{x}\left(\int_0^{\tau_b} \int_{(-\infty,0)} \U(r+s_0, X_r+y)\I_{\{ X_{r} > \ct(r+s_0) \}}\Pi( \dd y)\dd r\right) \leq 0.
 \end{align*}
 However, the continuity of the functions $\b$ and $\ct$ gives the existence of $s_1\in (s_0,\m)$ such that $ \ct(r)<\b(r)$ for all $r\in [s_0,s_1]$. Hence, together with the fact that $\int_{(-\infty,0)} \U(x+y,t)\Pi(\dd y)+\G(x,t)>0$ for all $(t,x)\in D_c$ we can conclude that 

\begin{align*}
 \E_{x}&\left(\int_0^{\tau_b}\G(r+s_0,X_r)\I_{\{ X_{r} \geq \ct(r+s_0) \}} \dd r\right)\\
&\qquad+\E_{x}\left(\int_0^{\tau_b} \int_{(-\infty,0)} \U(r+s_0, X_r+y)\I_{\{ X_{r} > \ct(r+s_0) \}}\Pi( \dd y)\dd r\right) > 0,
 \end{align*}
which shows a contradiction.

\end{proof}

\section{Examples}
\label{sec:examples}

\subsection{Brownian motion with drift}

Suppose that $X=\{X_t ,t\geq 0 \}$ is a Brownian motion with drift. That is for any $t\geq 0$, $X_t=\mu t+\sigma B_t$, where $\sigma>0$ and $\mu\in \R$. In this case, we have that 
\begin{align*}
\psi(\beta)=\mu \beta +\frac{1}{2}\sigma^2 \beta^2
\end{align*}
for all $\beta \geq 0$. Then 
\begin{align*}
\Phi(q)=\frac{1}{\sigma^2}\left[\sqrt{\mu^2+2\sigma^2 q }-\mu \right].
\end{align*}
It is well known that $-\underline{X}_{\et}$ has exponential distribution (see e.g. \cite{Borodin_2002} pp251 or \cite{kyprianou2014fluctuations} pp 233) with distribution function given by
\begin{align*}
F^{(\theta)}(x)=1-\exp\left(-\frac{x}{\sigma^2}\left[\sqrt{\mu^2+2\sigma^2 \theta }+\mu \right] \right)\qquad \text{for } x>0.
\end{align*}
Denote $\Phi(x;a,b^2)$ as the distribution function of a Normal random variable with mean $a \in \R$ and variance $b^2$, i.e., for any $x\in \R$,
\begin{align*}
\Phi(x;a, b^2)= \int_{-\infty}^x  \frac{1}{\sqrt{2\pi b^2}} e^{-\frac{1}{2b^2} (y-a)^2}\dd y.
\end{align*}
For any $b,s,t\geq 0$ and $x\in \R$, define the function 
\begin{align*}
K(t,x,s,b)&=\E\left(  G^{(\theta)}(s+t, X_s+x)\I_{\{ X_s+x\leq b\}}\right).
\end{align*}
Then it can be easily shown that 
\begin{align*}
K(t,x,s,b) 
&= \Phi(b-x;\mu s,\sigma^2 s)-2e^{-\theta (s+t)}\Phi(-x;\mu s,\sigma^2 s)-2e^{-\theta t} \exp\left(-\frac{x}{\sigma^2}\left[\sqrt{\mu^2+2\sigma^2 \theta }+\mu \right] \right)\\
&\qquad \times \left[  \Phi(b-x,-s\sqrt{\mu^2+2\sigma^2 \theta},s \sigma^2)-\Phi(-x,-s\sqrt{\mu^2+2\sigma^2 \theta},s \sigma^2) \right].
\end{align*}
Thus, we have that $b^{(\theta)}$ satisfies the non-linear integral equation
\begin{align*}
\int_0^{\m-t} K(t,b^{(t)}(t),s,b^{(\theta)}(t+s))\dd s=0
\end{align*}
for all $t\in [0,\m)$ and the value function $\V$ is given by
\begin{align*}
\V(t,x)=\int_0^{\m-t} K(t,x,s, \b(t+s)) \dd s
\end{align*}
for all $(t,x)\in \R_+ \times \R$. Note that we can approximate the integrals above by Riemann sums so a numerical approximation can be implement. Indeed, take $n \in \mathbb{Z}_+$ sufficiently large and define $h=\m/n$.  For each $k\in \{0,1,2,\ldots,n \}$, we define $t_k=kh$. Then the sequence of times $\{ t_k, k= 0,1,\ldots,n \}$ is a partition of the interval $[0,\m]$. Then, for any $x\in \R$ and $t\in [t_k,t_{k+1})$ for $k \in \{0,1,\ldots, n-1 \}$ we approximate $\V(t,x)$ by 
\begin{align*}
\V_h(t_k,x)=\sum_{i=k}^{n-1} K(t_k,x,t_{i-k+1},b_{i})h,
\end{align*}
where the sequence $\{b_k, k=0,1,\ldots,n -1\}$ is a solution to 
\begin{align*}
\sum_{i=k}^{n-1} K(t_k,x,t_{i-k+1},b_{i})=0
\end{align*}
for each $k\in \{0,1,\ldots, n-1 \}$. Note that the sequence $\{b_k, k=0,1,\ldots,n \}$ is a numerical approximation to the sequence $\{\b(t_k), k=0,1,\ldots,n-1 \}$ (for $n$ sufficiently large) and can be calculated by using backwards induction. In the Figure \ref{fig:Brownianmotionmu2sigma1}, we show a numerical calculation of the equations above. The parameters used are $\mu=2$ and $\sigma=1$, whereas we chose $\m=10$.

\begin{figure}[htbp]
\centering
\includegraphics[scale=0.4]{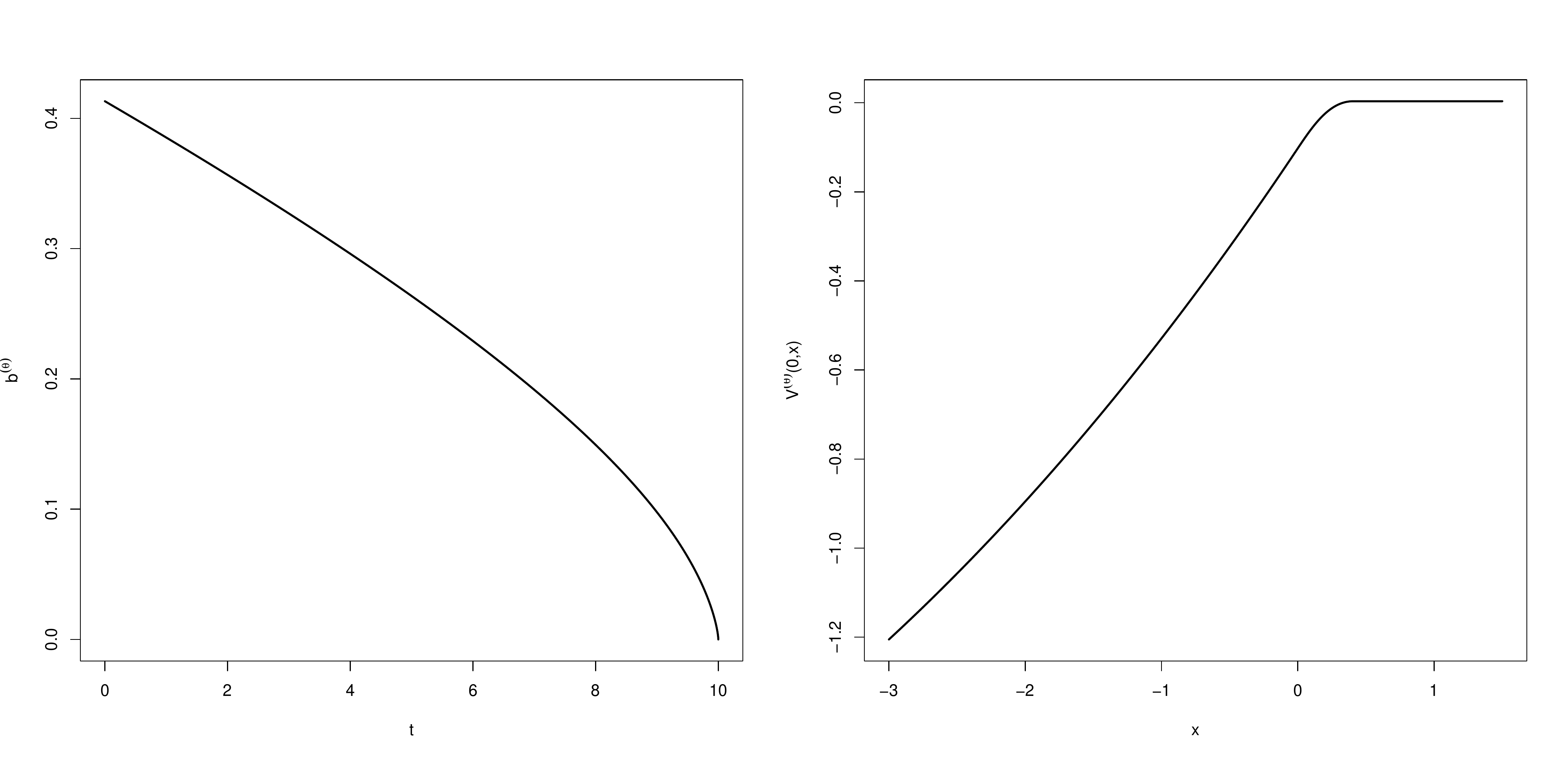}
\caption{Brownian motion with drift $\mu=2$ and $\sigma=1$. Left hand side: Optimal boundary; Right hand side: Value function fixing $t=1$.}
\label{fig:Brownianmotionmu2sigma1}
\end{figure}

\newpage
\subsection{Brownian motion with exponential jumps}

Let $X=\{X_t,t\geq 0  \}$ be a compound Poisson process perturbed by a Brownian motion, that is 
\begin{align}
\label{eq:Brownianmotionwithexpjumps}
X_t=\sigma B_t+\mu t -\sum_{i=1}^{N_t} Y_i,
\end{align}
where $B=\{B_t,t\geq 0 \}$ is a standard Brownian motion, $N=\{N_t,t\geq 0 \}$ is Poisson process with rate $\lambda$ independent of $B$, $\mu \in \R$, $\sigma > 0$ and the sequence $\{Y_1,Y_2,\ldots \}$ is a sequence of independent random variables exponentially distributed with mean $1/\rho>0$. Then in this case, the Laplace exponent is derived as
\begin{align*}
\varphi(\beta)=\frac{\sigma^2}{2}\beta^2 +\mu \beta -\frac{\lambda \beta}{\rho +\beta }.
\end{align*}
Its L\'evy measure, given by $\Pi(\dd y)=\lambda \rho e^{\rho y} \I_{\{y<0 \}} \dd y$ is a finite measure and $X$ is a process of infinite variation. According to \cite{kyprianou2011theory}, the scale function in this case is given for $q\geq 0$ and $x\geq 0$ by,
\begin{align*}
W^{(q)}(x)=\frac{e^{\Phi(q)x}}{\psi'(\Phi(q))}+\frac{e^{\zeta_1(q) x}}{\psi'(\zeta_1(q))} +\frac{e^{\zeta_2(q) x}}{\psi'(\zeta_2(q))}, 
\end{align*}
where $\zeta_2(q), \zeta_1(q)$ and $\Phi(q)$ are the three real solutions to the equation $\psi(\beta)=q$, which satisfy $\zeta_2(q)<-\rho<\zeta_1(q)<0<\Phi(q)$. The second scale function, $Z^{(q)}$, takes the form
\begin{align*}
Z^{(q)}(x)
&=1+q  \left[  \frac{e^{\Phi(q)x}-1}{\Phi(q)\psi'(\Phi(q))}+\frac{e^{\zeta_1(q) x}-1}{\zeta_1(q)\psi'(\zeta_1(q))} +\frac{e^{\zeta_2(q) x}-1}{\zeta_2(q)\psi'(\zeta_2(q))} \right].
\end{align*}
Note that since we have exponential jumps (and hence $\Pi(\dd y)=\lambda \rho e^{\rho y} \I_{\{y<0\}}$), we have that for all $t\in [0,\m)$ and $x>0$,
\begin{align*}
\int_{(-\infty,-x)} \V(t,\b(t)+x+y)\Pi(\dd y)
&=e^{-\rho x }\int_{(-\infty,0)} \V(t,\b(t)+y)\Pi(\dd y).
\end{align*}
Then, for any $(t,x)\in [0,\m]\times \R$, equation \eqref{eq:characterizationofV} reads as
\begin{align*}
\V(t,x)
&=\E_x\left( \int_0^{\m-t	} \G(r+t,X_r)\I_{\{X_r <\b(r+t) \}}\dd r \right)\nonumber \\
&\qquad-\E_x\left( \int_0^{\m-t} e^{-\rho(X_r -\b(r+t)) }\mathcal{V}(r+t) \I_{\{X_r> \b(r+t) \}}\dd r \right)
\end{align*}
where for any $r,s\in [0,\m)$, $b\geq 0$ and $x\in \R$,
\begin{align*}
\mathcal{V}(t)&=\int_{(-\infty,0)} \V(t,\b(t)+y)\Pi(\dd y).
\end{align*}
Note that the equation above suggest that in order to find a numerical value of $\b$ using Theorem \ref{thm:Vsatisfiesequation} we only need to know the values of the function $\mathcal{V}$ and not the values of $\int_{(-\infty,0)}\V(t,x+y)\Pi(\dd y)$ for all $t\in [0,\m]$ and $x>\b(t)$. The next Corollary confirms that notion.
\begin{cor}
Let $\theta>0$. Assume that $X=\{X_t, t\geq 0 \}$ is of the form \eqref{eq:Brownianmotionwithexpjumps} with $\mu\in \R$, $\sigma, \lambda, \rho>0$. Suppose that $\ct$ and $\mathcal{U}$ are continuous functions on $[0,\m)$ such that $\ct \geq h^{(\theta)}$ and $0\geq \mathcal{U}(t) \geq -\G(t,\ct(t))$ for all $t\in [0,\m)$. For any $(t,x)\in [0,\m]\times \R$ we define the function 
\begin{align*}
\U(t,x)&=\E_x\left( \int_0^{\m-t	} \G(r+t,X_r)\I_{\{X_r <\ct(r+t) \}}\dd r \right) \\
&\qquad-\E_x\left( \int_0^{\m-t} e^{-\rho(X_r -\ct(r+t)) }\mathcal{U}(r+t) \I_{\{X_r> \ct(r+t) \}}\dd r \right).
\end{align*}
Further assume that there exists a value $h>0$ such that $\U(t,x)=0$ for any $t\in [0,\m)$ and $x\in [\b(t),\b(t)+h]$. If $\U$ is a non-positive function, we have that $\ct= \b$ and $\U=\V$.
\end{cor}
\begin{proof}
First note that, since $X$ is of infinite variation, $\P_x(X_r=\ct(r+t))=0$ for all $r,t \in [0,\m)$ such that $r+t<\m$ and $x\in \R$. Hence, by continuity of $\G$ and $\mathcal{U}$, and by dominated convergence theorem, we have that $\U$ is continuous. By means of Theorem \ref{thm:Vsatisfiesequation} is enough to show that $\U$ satisfies the integral equation,
\begin{align*}
\U(t,x)&=\E_x\left( \int_0^{\m-t	} \G(r+t,X_r)\I_{\{X_r <\ct(r+t) \}}\dd r \right) \\
&\qquad-\E_x\left( \int_0^{\m-t} \int_{(-\infty,\ct(r+t)-X_r)} \U(r+t,X_r+y)\Pi(\dd y) \I_{\{X_r> \ct(r+t) \}}\dd r \right)\\
&=\E_x\left( \int_0^{\m-t	} \G(r+t,X_r)\I_{\{X_r <\ct(r+t) \}}\dd r \right) \\
&\qquad-\E_x\left( \int_0^{\m-t} e^{-\rho(X_r -\ct(r+t)) }H(r+t) \I_{\{X_r>\ct(r+t) \}}\dd r \right),
\end{align*}
where $H(r)=\int_{(-\infty,0)} \U(r,\ct(r)+y)\Pi(\dd y)$ for all $r\in [0,\m)$ and in the last equality we used the explicit form of $\Pi(\dd y)$. Then it suffices to show that $H(t)=\mathcal{U}(t)$ for all $t\in [0,\m)$.\\

Let $t\geq 0$. For any $\delta \in (0,\m-t)$, consider the stopping time 
\begin{align*}
\tau_{\delta}=\inf\{ s\in [0,\delta]: X_s\notin [\ct(s+t), \ct(s+t)+h]\}.
\end{align*}
Note that for any $s<\tau_{\delta}$ we have that $X_s\in (\ct(s+t), \ct(s+t)+h)$ and $X_{\delta}\in (\ct(s+t), \ct(s+t)+h)$ in the event $\{  \tau_{\delta}=\delta \}$. Then using the strong Markov property at time $\tau_{\delta}$, we have that for any $x\in [\ct(t), \ct(t)+h)$, 
\begin{align*}
0&=\U(t,x)
= -\E_x\left( \int_0^{\tau_{\delta}} e^{-\rho(X_r -\ct(r+t)) }\mathcal{U}(r+t) \dd r \right)+\E_x( \U(t+\tau_{\delta},X_{\tau_{\delta}})\I_{\{X_{\tau_{\delta}}<\ct(t+\tau_{\delta})\} }),
\end{align*}
where in the last equality we used the fact that $\U(t,x)=0$ for all $x\in [\ct(t), \ct(t)+h]$, the continuity of $\ct$ and that $X$ can only cross above $\ct$ by creeping. By using the compensation formula for Poisson Random measures, we obtain that 
\begin{align*}
\E_x( \U(t+\tau_{\delta},X_{\tau_{\delta}})\I_{\{X_{\tau_{\delta}}<\ct(t+\tau_{\delta})\} })
&=\E_x\left( \int_{0}^{\tau_{\delta}} \int_{(-\infty,0)}  \U(t+r,X_{r}+y) \I_{\{X_{r}+y<\ct(t+r)\} }\dd r \Pi( \dd y)\right)\\
&=\E_x\left( \int_0^{\tau_{\delta}} e^{-\rho(X_r -\ct(r+t)) }H(r+t) \dd r \right).
\end{align*}
Hence we conclude that for any $\delta>0$, 
\begin{align}
\label{eq:equationintermsofHandmathcalU}
0&=\E_x\left( \int_0^{\tau_{\delta}} e^{-\rho(X_r -\ct(r+t)) }[H(r+t)-\mathcal{U}(r+t)] \dd r \right)
\end{align}
and hence,
\begin{align*}
0\leq \E_x\left( \int_0^{\tau_{\delta}} [H(r+t)-\mathcal{U}(r+t)] \dd r \right).
\end{align*}
By continuity of $H$ and $\mathcal{U}$ we obtain that
\begin{align*}
0\leq \lim_{\delta \downarrow 0 }\frac{1}{\E_x(\tau_{\delta})}\E_x\left( \int_0^{\tau_{\delta}} [H(r+t)-\mathcal{U}(r+t)] \dd r \right)=H(t)-\mathcal{U}(t).
\end{align*}
Moreover, from equation \eqref{eq:equationintermsofHandmathcalU} we conclude that $H(t)=\mathcal{U}(t)$ for all $t\in [0,\m)$ and the conclusion holds. 
\end{proof}
Hence, for any $(t,x)\in [0,\m)\times \R$, we can write
\begin{align*}
\V(t,x)
&=\int_0^{\m-t}	K_1(t,x,r,\b(r+t))\dd r -\int_0^{\m-t} \mathcal{V}(r+t)  K_2(t,x,r,\b(r+t))\dd r,
\end{align*}
where for any $r,s\in [0,\m)$, $b\geq 0$ and $x\in \R$,
\begin{align*}
\mathcal{V}(t)&=\int_{(-\infty,0)} \V(t,\b(t)+y)\Pi(\dd y)\\
K_1(t,x,s,b)&= \E\left(  \G(s+t,X_s+x)\I_{\{X_s <b-x \}}\right)\\
K_2(t,x,s,b)&=  \E\left(  e^{-\rho(X_s+x -b) } \I_{\{X_s> b-x \}} \right).
\end{align*}
Take a value $h_0>0$ sufficiently small. Hence the functions $\b$ and $\mathcal{V}$ satisfy the integral equations
\begin{align*}
\int_0^{\m-t}	&K_1(t,\b(t),r,\b(r+t))\dd r  -\int_0^{\m-t} \mathcal{V}(r+t,\b(r+t))  K_2(t,\b(t),r,\b(r+t))\dd r =0\\
\int_0^{\m-t}	&K_1(t,\b(t)+h_0,r,\b(r+t))\dd r  -\int_0^{\m-t} \mathcal{V}(r+t,\b(r+t))  K_2(t,\b(t)+h_0,r,\b(r+t))\dd r =0
\end{align*}
for all $t\in [0,\m]$. We can approximate the integrals above by Riemann sums so a numerical approximation can be implement. Indeed, take $n \in \mathbb{Z}_+$ sufficiently large and define $h=\m/n$.  For each $k\in \{0,1,2,\ldots,n\}$, we define $t_k=kh$. Then the sequence of times $\{ t_k, k= 0,1,\ldots,n \}$ is a partition of the interval $[0,\m]$. Then, for any $x\in \R$ and $t\in [t_k,t_{k+1})$, we approximate $\V(t,x)$ by 
\begin{align*}
\V_h(t_k,x)=\sum_{i=k}^{n-1} [K_1(t_k,x,t_{i-k+1},b_{i})-\mathcal{V}_iK_2(t_k,x,t_{i-k+1},b_{i}) ]h,
\end{align*}
where the sequence $\{(b_k, \mathcal{V}_k), k=0,1,\ldots,n-1 \}$ is a solution to 
\begin{align*}
\sum_{i=k}^{n-1} [K_1(t_k,b_k,t_{i-k+1},b_{i})-\mathcal{V}_{i}K_2(t_k,b_k,t_{i-k+1},b_{i}) ]=0\\
\sum_{i=k}^{n-1} [K_1(t_k,b_k+h_0,t_{i-k+1},b_{i})-\mathcal{V}_{i}K_2(t_k,b_k+h_0,t_{i-k+1},b_{i}) ]=0
\end{align*}
for each $k\in \{0,1,\ldots, n -1\}$. Note that, for $n$ sufficiently large, the sequence $\{(b_k,\mathcal{V}_k), k=0,1,\ldots,n \}$ is a numerical approximation to the sequence $\{(\b(t_k), \mathcal{V}(t_k)), k=0,1,\ldots,n \}$ (provided that $\V_h\leq 0$) and can be calculated by using backwards induction. The functions $K_1$ and $K_2$ can be estimated using simulation methods. In the figures below we include a plot of the numerical calculation of $\b$ and $\V(0,x)$ using the parameters $\theta=\log(2)/10$, $\mu=3$, $\sigma=\lambda=\rho=1$. 

\begin{figure}[htbp]
\centering
\includegraphics[scale=0.3]{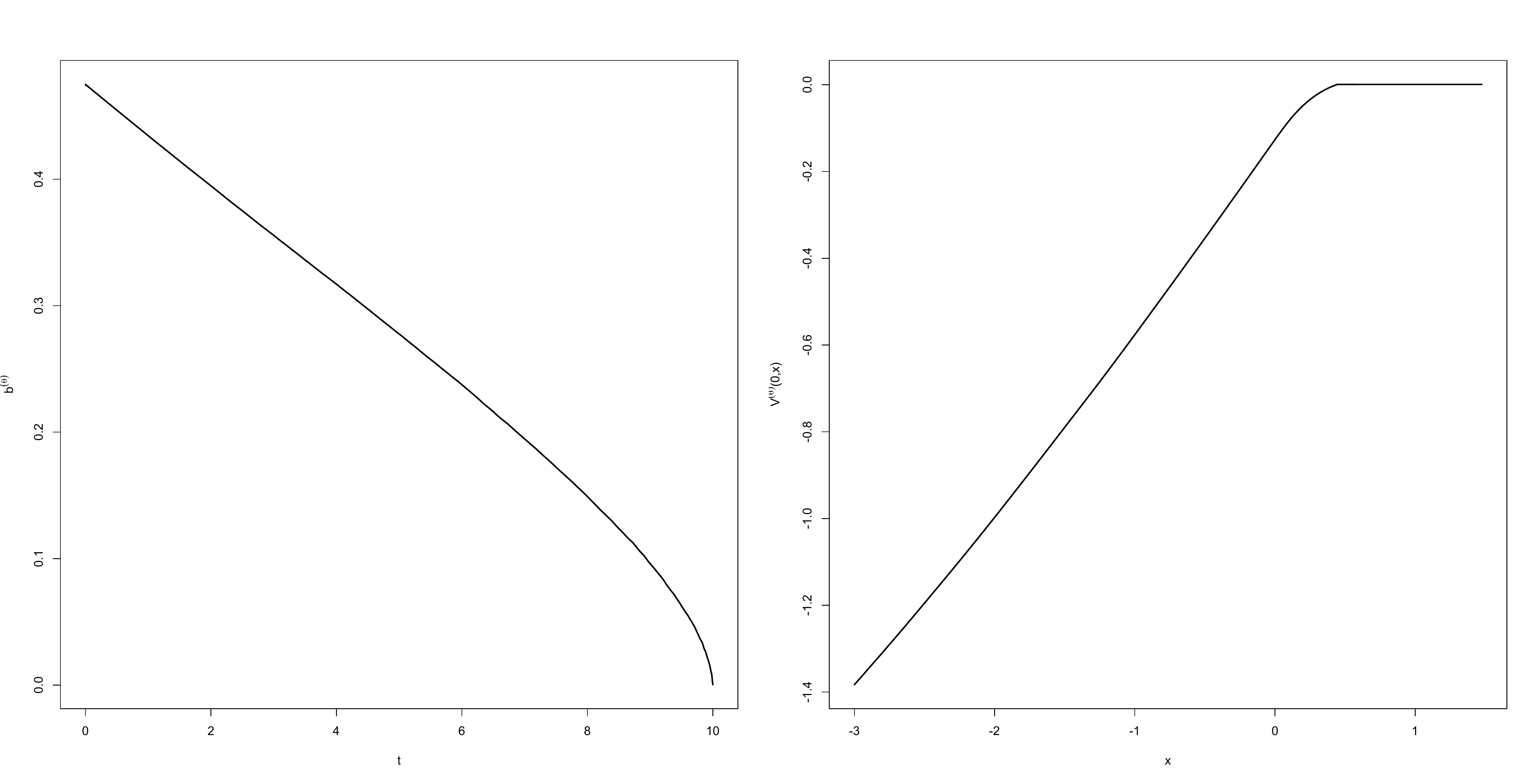}
\caption{Brownian motion with drift perturbed by a compound Poisson process with exponential sized jumps with $\mu=3$ and $\sigma=\lambda=\rho=1$. Left hand side: Optimal boundary; Right hand side: Value function fixing $t=0$.}
\label{fig:Brownianmotionmu2sigma1}
\end{figure}

%
%
%


\newpage
\bibliography{bibexponentialtime}
\bibliographystyle{apalike}
\nocite{*}

\end{document}